\documentclass{birkjour_t2}
\usepackage{amsmath}
\usepackage{amsfonts}
\usepackage{amssymb}
\usepackage[english]{babel}
\usepackage{color}
\usepackage{graphicx}
\usepackage{lmodern}
\usepackage{amsthm}
\usepackage{mathrsfs}
\usepackage{microtype}
\usepackage{mathscinet}
\usepackage{enumitem}
\usepackage[cal=boondoxo,bb=ams]{mathalfa}
\usepackage{hyperref}
\hypersetup{hidelinks}

\renewcommand{\theequation}{\arabic{section}.\arabic{equation}}

\newtheorem{thm}{Theorem}[section]

\newtheorem{lem}{Lemma}[section]

\newtheorem{rem}{Remark}[section]

\newcommand{\ml}{\mathcal}
\newcommand{\mb}{\mathbb}
\DeclareMathOperator{\diag}{diag}
\DeclareMathOperator{\divv}{div}
\DeclareMathOperator{\intt}{int}
\DeclareMathOperator{\extt}{ext}
\DeclareMathOperator{\midd}{mid}

\begin{document}

%-------------------------------------------------------------------------
% editorial commands: to be inserted by the editorial office
%
%\firstpage{1} \volume{228} \Copyrightyear{2004} \DOI{003-0001}
%\seriesextraline{This is the Concrete Title of this Book\br H.E. R and S.T.C. W, Eds.}
%\firstpage{1}
%      environment-name

\title[Elastic waves with Kelvin-Voigt damping]
 {Decay properties and asymptotic profiles for elastic waves with Kelvin-Voigt damping in 2D}

%----------Author 1
\author[W. Chen]{Wenhui Chen}
\address{Institute of Applied Analysis, Faculty for Mathematics and Computer Science\\
	 Technical University Bergakademie Freiberg\\
	  Pr\"{u}ferstra{\ss}e 9\\
	   09596 Freiberg\\
	    Germany}
\email{wenhui.chen.math@gmail.com}

%\MSC[2010] Primary 35L52, 35L99; Secondary 35B33, 35B44

\date{January 1, 2004}

\begin{abstract}
In this paper we consider elastic waves with Kelvin-Voigt damping in 2D. For the linear problem, applying pointwise estimates of the partial Fourier transform of solutions in the Fourier space and asymptotic expansions of eigenvalues and their eigenprojections, we obtain sharp energy decay estimates with additional $L^m$ regularity and $L^p-L^q$ estimates on the conjugate line. Furthermore, we derive asymptotic profiles of solutions under different assumptions of initial data. For the semilinear problem, we use the derived $L^2-L^2$ estimates with additional $L^m$ regularity to prove global (in time) existence of small data solutions to the weakly coupled system. Finally, to deal with elastic waves with Kelvin-Voigt damping in 3D, we apply the Helmholtz decomposition.
\end{abstract}

%%% ----------------------------------------------------------------------
\maketitle
%%% ----------------------------------------------------------------------
%\tableofcontents
\section{Introduction}\label{Introduction}
Elastic waves describe particles vibrate in the material holding the property of elasticity. Moreover, when the particles are moving and there exists a force acts on the particles restore them to their original position, elastic waves will be produced. For the sake of briefness, the system governing motion of an infinite, isotropic and homogeneous elastic continuum are given by
\begin{equation}\label{101}
\left\{
\begin{aligned}
&u_{tt}-a^2\Delta u-\left(b^2-a^2\right)\nabla\divv u=0,&&t>0,\,\,x\in\mb{R}^n,\\
&(u,u_t)(0,x)=(u_0,u_1)(x),&&x\in\mb{R}^n,
\end{aligned}
\right.
\end{equation}
where the unknowns $u=u(t,x)\in\mb{R}^n$ for $n=2,3,$ denote the elastic displacement, $t$ stands for the time and $x$ stands for the space-variable. The positive constants $a$ and $b$ are related to the Lam\'e constants satisfying $b>a>0$. The system of elastic waves satisfies the property of finite speed of propagation given by coefficient $b$, which is the speed of propagation of the longitudinal $P$-wave. The coefficient $a$ is the speed of propagation of the transverse $S$-wave. This property have been discussed in the paper \cite{CharaoMenzala1992}.

In general, we cannot expect that the system \eqref{101} models real-world problems, since there exist several kinds resistance, such as fluid resistance and  frictional resistance. People always use damping mechanisms to describe oscillation amplitude, which are reduced through the irreversible removal of the vibratory energy in a mechanical system or a component (c.f. \cite{AkayCarcaterra}).

We recall some results for elastic waves with different damping mechanisms. In recent years, there are some papers devoted the study of linear elastic waves with \emph{friction} ($\theta=0$) or \emph{structural damping} ($\theta\in(0,1]$) as follows:
\begin{equation}\label{linear elastic waves with damping}
\left\{
\begin{aligned}
&u_{tt}-a^2\Delta u-\left(b^2-a^2\right)\nabla\divv u+(-\Delta)^{\theta}u_t=0,&&t>0,\,\,x\in\mb{R}^n,\\
&(u,u_t)(0,x)=(u_0,u_1)(x),&&x\in\mb{R}^n.
\end{aligned}
\right.
\end{equation}
In the paper \cite{IkehataCharaodaLuz2014}, the authors considered the Cauchy problem \eqref{linear elastic waves with damping} with $\theta\in[0,1]$ and initial data taking from $\left(H^{s+1}(\mb{R}^n)\cap L^1(\mb{R}^n)\right)\times\left(H^s(\mb{R}^n)\cap L^1(\mb{R}^n)\right)$ and derived almost sharp energy estimates in $n$ dimensional space, $n\geq2$, by using the energy method in the Fourier space. For the Cauchy problem \eqref{linear elastic waves with damping} in two dimensions with $\theta\in(0,1]$, the author of \cite{Reissig2016} studied qualitative properties of solutions, including Gevrey smoothing if $\theta\in(0,1)$, propagation of singularities if $\theta=1$ and estimates of higher-order energies. In the recent paper \cite{ChenReissig2018}, smoothing effect, energy estimates with initial data belonging to different function spaces and diffusion phenomena for solutions to the Cauchy problem \eqref{linear elastic waves with damping} in three dimensions with $\theta\in[0,1]$ are investigated by applying diagonalization procedure (see \cite{JachmannReissig2008,ReissigWang2005,Yagdjian1997} for further explanations about this method) and the energy method in the Fourier space.

More general structural damping mechanisms are considered in \cite{WuChaiLi2017}. Especially, they studied energy estimates for the following elastic waves with Kelvin-Voigt damping:
\begin{equation*}
\left\{
\begin{aligned}
&u_{tt}-a^2\Delta u-\left(b^2-a^2\right)\nabla\divv u+\left(-a^2\Delta -\left(b^2-a^2\right)\nabla \divv\right)u_t=0,&&t>0,\,\,x\in\mb{R}^n,\\
&(u,u_t)(0,x)=(u_0,u_1)(x),&&x\in\mb{R}^n,
\end{aligned}
\right.
\end{equation*}
for $n\geq2$. By applying the Haraux-Komornik inequality and the energy method in the Fourier space, the authors of \cite{WuChaiLi2017} proved almost sharp decay estimates of the total energy. Nevertheless, sharp estimate of the energy and asymptotic profiles of the solution are still not investigated.

Other studies on the dissipative elastic waves can be found in the literature. We refer to \cite{CharaoIkehata2007,CharaoIkehata2011,HorbachNakabayashi2014} for elastic waves with a classical damping term containing time and spatial variables, \cite{JachmannReissig2008,JachmannReissig2009,Racke1990,RackeWang1998,ReissigWang199901,ReissigWang199902,Wang2002} for elastic waves with thermal dissipative.

The present paper is devoted to the study of the following linear elastic waves with \emph{Kelvin-Voigt damping} (c.f. \cite{MunozRiveraRacke2017,LiuChenLiu1998}) in 2D:
\begin{equation}\label{linearproblem Kelvin-Voigt 2D}
\left\{
\begin{aligned}
&u_{tt}-a^2\Delta u-\left(b^2-a^2\right)\nabla\divv u+\left(-a^2\Delta -\left(b^2-a^2\right)\nabla \divv\right)u_t=0,&&t>0,\,\,x\in\mb{R}^2,\\
&(u,u_t)(0,x)=(u_0,u_1)(x),&&x\in\mb{R}^2,
\end{aligned}
\right.
\end{equation}
where $u=\big(u^{(1)},u^{(2)}\big)^{\mathrm{T}}$, and the weakly coupled system of semilinear elastic waves with Kelvin-Voigt damping in 2D, that is, \eqref{linearproblem Kelvin-Voigt 2D} with nonlinear terms $f(u):=\left(|u^{(2)}|^{p_1},|u^{(1)}|^{p_2}\right)^{\mathrm{T}}$ on the right-hand sides (one can see \eqref{semilinear problem Kelvin-Voigt 2D} later). Considering the case $a^2=b^2$ in \eqref{linearproblem Kelvin-Voigt 2D}, the model will be transferred to the wave equation with viscoelastic damping. Then, one can apply partial Fourier transform and derive estimates of solutions and asymptotic profiles of solutions (c.f. \cite{DabbiccoReissig2014,ReissigEbert2018, Ikehata2014,Michihisa2018}). However, due to the Lam\'e operator $\left(-a^2\Delta-\left(b^2-a^2\right)\nabla\divv\right)$ with $b^2>a^2>0$ appearing in the model, we cannot directly follow the methods from \cite{DabbiccoReissig2014,ReissigEbert2018, Ikehata2014,Michihisa2018} to get the sharp energy estimates of an energy and asymptotic profiles for the solutions. Moreover, in this paper we also deal with the three dimensional elastic waves with Kelvin-Voigt damping by applying the Helmholtz decomposition.

The paper is organized as follows. We first investigate some qualitative properties of solutions to the two dimensional linear Cauchy problem \eqref{linearproblem Kelvin-Voigt 2D} from Section \ref{Section KV 2D estimate of solution} to Section \ref{Section KV 2D refinement of decay estimates}. More precisely, we prepare the pointwise estimates of the partial Fourier transform of solutions in the Fourier space by applying Lemma 2.2 in \cite{UmedaKawashimaShizuta1984} in Section \ref{Section KV 2D estimate of solution}. Then, we obtain energy estimates of solutions to the dissipative system \eqref{linearproblem Kelvin-Voigt 2D} with initial data having additional $L^m$ regularity, and $L^p-L^q$ estimates on the conjugate line. In Subsection \ref{Section KV 2D asymptic expansion} we derive the asymptotic expansions of eigenvalues and their eigenprojections to show the sharpness of previous pointwise estimates. After constructing asymptotic representation in Subsection \ref{Section KV 2D asymptic expresssion}, the asymptotic profiles of solutions are derived in Section \ref{Section KV 2D refinement of decay estimates}, where initial data taking from Bessel potential spaces or weighted $L^1$ spaces. Then, in Section \ref{global existence small data solution KV 2D} we prove global (in time) existence of small data solutions to the weakly coupled system of semilinear elastic waves with Kelvin-Voigt damping. Additionally, the three dimensional elastic waves with Kelvin-Voigt damping is treated in Section \ref{3D model section}. Finally, in Section \ref{concluding remarks section} some concluding remarks complete the paper.
\medskip

\textbf{Notations:} Throughout this paper we will use the following notations. Let $\ml{F}(g)=\hat{g}$ denote the Fourier transform of $g$ and let $\ml{F}^{-1}(g)=\check{g}$ denote the inverse Fourier transform. We denote  by $\ml{S}(\mb{R}^n)$ Schwartz space of rapidly decay function. For $s\in\mb{R}$ and $1\leq p<\infty$, we denote by $H^s_p(\mb{R}^n)$ and $\dot{H}^s_p(\mb{R}^n)$ Bessel and Riesz potential spaces based on $L^p(\mb{R}^n)$ spaces. Furthermore, $\langle D\rangle^s$ and $|D|^s$ stand for the pseudo-differential operators with the symbols $\langle \xi\rangle^s$ and $|\xi|^s$, respectively. Here $\langle \xi\rangle^2:=1+|\xi|^2$. For $\gamma\in[0,\infty)$, the weighted function space $L^{1,\gamma}(\mb{R}^n)$ is defined by
\begin{equation*}
L^{1,\gamma}(\mb{R}^n):=\left\{g\in L^1(\mb{R}^n):\|g\|_{L^{1,\gamma}(\mb{R}^n)}:=\int_{\mb{R}^n}(1+|x|)^{\gamma}|g(x)|dx<\infty\right\}.
\end{equation*}
Let us define the function spaces \begin{equation*}
\ml{A}_{m,s}(\mb{R}^n):=\left(H^{s+1}(\mb{R}^n)\cap L^m(\mb{R}^n)\right)\times \left(H^s(\mb{R}^n)\cap L^m(\mb{R}^n)\right)
\end{equation*} for $s\geq0$ and $m\in[1,2]$ carrying the corresponding norm
\begin{equation*}
\|(g,h)\|_{\ml{A}_{m,s}(\mb{R}^n)}:=\|g\|_{H^{s+1}(\mb{R}^n)}+\|g\|_{L^m(\mb{R}^n)}+\|h\|_{H^{s}(\mb{R}^n)}+\|h\|_{L^m(\mb{R}^n)}.
\end{equation*}
Let $\chi_{\intt},\chi_{\midd},\chi_{\extt}\in \ml{C}^{\infty}\left(\mb{R}^2\right)$ having their supports in $Z_{\intt}(\varepsilon):=\left\{\xi\in\mb{R}^2:|\xi|<\varepsilon\right\}$, $Z_{\midd}(\varepsilon):=\left\{\xi\in\mb{R}^2:\varepsilon\leq|\xi|\leq\frac{1}{\varepsilon}\right\}$ and $Z_{\extt}(\varepsilon):=\left\{\xi\in\mb{R}^2:|\xi|>\frac{1}{\varepsilon}\right\}$, respectively, so that $\chi_{\midd}=1-\chi_{\intt}-\chi_{\extt}$. Here $\varepsilon>0$ is a small constant. We write $g\lesssim h$, when there exists a constant $C>0$ such that $g\leq Ch$. 

\section{Estimates of solutions to the linear Cauchy problem}\label{Section KV 2D estimate of solution}
In this section we will derive estimates of solutions to the linear elastic waves with Kelvin-Voigt damping in two dimensions.

Indeed, the system \eqref{linearproblem Kelvin-Voigt 2D} can be rewritten by the following form:
\begin{equation}\label{sss111 2D}
\begin{split}
u_{tt}-a^2\Delta (u+u_t)-\left(b^2-a^2\right)\left(
{\begin{array}{*{20}c}
	\partial_{x_1}^2 & \partial_{x_1x_2}\\
	\partial_{x_2x_1} & \partial_{x_2}^2 \\
	\end{array}}
\right)(u+u_t)=0.
\end{split}
\end{equation}
Applying the partial Fourier transformation with respect to spatial variable to \eqref{sss111 2D}, i.e.,  $\hat{u}(t,\xi)=\ml{F}_{x\rightarrow \xi}(u(t,x))$ gives
\begin{equation}\label{40002 2D}
\hat{u}_{tt}+|\xi|^{2}A(\eta)\hat{u}_t+|\xi|^2A(\eta)\hat{u}=0,
\end{equation}
where $\eta=\xi/|\xi|\in\mb{S}^1$ and
\[
\begin{split}
A(\eta)=\left(
{\begin{array}{*{20}c}
	a^2+(b^2-a^2)\eta_1^2 & (b^2-a^2)\eta_1\eta_2\\
	(b^2-a^2)\eta_1\eta_2 & a^2+(b^2-a^2)\eta_2^2\\
	\end{array}}
\right).
\end{split}
\]
Because of our assumption $b^2>a^2>0$, the matrix $A(\eta)$ is positive definite. The eigenvalues of $A(\eta)$ are $b^2$ and $a^2$. We introduce the matrices
\[
\begin{split}
M(\eta)=\left(
{\begin{array}{*{20}c}
	\eta_1 & \eta_2\\
	\eta_2 & -\eta_1 \\
	\end{array}}
\right) \,\,\,\,\text{and}\,\,\,\, A_{\text{diag}}(|\xi|)=|\xi|^2\text{diag}\left(b^2,a^2\right).
\end{split}
\]
After the change of variables $v(t,\xi):={M}^{-1}(\eta)\hat{u}(t,\xi)$, we define
\begin{equation*}
W(t,\xi):=\left(
\begin{aligned}
&v_t(t,\xi)+i{A}_{\text{diag}}^{1/2}(|\xi|)v(t,\xi)\\
&v_t(t,\xi)-i{A}_{\text{diag}}^{1/2}(|\xi|)v(t,\xi)
\end{aligned}\right).
\end{equation*}
By the above matrices we can directly compute 
\begin{equation*}
\begin{split}
W_t&=
\left(
\begin{aligned}
&-|\xi|^2{M}^{-1}(\eta){A}(\eta){M}(\eta)(v_{t}+v)+i{A}_{\text{diag}}^{1/2}(|\xi|)v_t\\
&-|\xi|^2{M}^{-1}(\eta){A}(\eta){M}(\eta)(v_{t}+v)-i{A}_{\text{diag}}^{1/2}(|\xi|)v_t
\end{aligned}\right)\\
&=\left(
\begin{aligned}
&-{A}_{\text{diag}}(|\xi|)(v_{t}+v)+i{A}_{\text{diag}}^{1/2}(|\xi|)v_t\\
&-{A}_{\text{diag}}(|\xi|)(v_{t}+v)-i{A}_{\text{diag}}^{1/2}(|\xi|)v_t
\end{aligned}\right)\\
&=\left(
{\begin{array}{*{20}c}
	-\frac{1}{2}{A}_{\text{diag}}(|\xi|)+i{A}^{1/2}_{\text{diag}}(|\xi|)&-\frac{1}{2}{A}_{\text{diag}}(|\xi|)\\
	-\frac{1}{2}{A}_{\text{diag}}(|\xi|)&-\frac{1}{2}{A}_{\text{diag}}(|\xi|)-i{A}^{1/2}_{\text{diag}}(|\xi|)\\
	\end{array}}
\right)W.
\end{split}
\end{equation*}
Therefore, we derive the first-order system
\begin{equation}\label{weshouldKelvinVoigt 2D}
\left\{\begin{aligned}
&W_t+\frac{1}{2}|\xi|^{2}{B}_0W-i|\xi|{B}_1W=0,&&t>0,\,\,\xi\in\mb{R}^2,\\
&W(0,\xi)=W_0(\xi),&&\xi\in\mb{R}^2,
\end{aligned}\right.
\end{equation}
where the coefficient matrices ${B}_0$ and ${B}_1$ are given by
\[
\begin{split}
{B}_0=\left(
{\begin{array}{*{20}c}
	b^2 & 0  & b^2 & 0 \\
	0  & a^2 &  0 & a^2\\
	b^2 & 0  & b^2 & 0 \\
	0 & a^2 & 0  & a^2\\
	\end{array}}
\right) \,\,\,\, \text{and} \,\,\,\, {B}_1=\text{diag}\left(b,a,-b,-a\right).
\end{split}
\]
The solution to \eqref{weshouldKelvinVoigt 2D} is given by 
\begin{equation*}
W(t,\xi)=e^{t\hat{\Phi}(|\xi|)}W_0(\xi),
\end{equation*}where
\begin{equation}\label{Se3 eigenvalue equation 2D}
\begin{split}
\hat{\Phi}(|\xi|)&=-\frac{1}{2}|\xi|^2{B}_0+i|\xi|{B}_1\\
&=\left(
{\begin{array}{*{20}c}
	-\frac{b^2}{2}|\xi|^2+ib|\xi| & 0  & -\frac{b^2}{2}|\xi|^2 & 0 \\
	0  & -\frac{a^2}{2}|\xi|^2+ia|\xi| &  0 & -\frac{a^2}{2}|\xi|^2\\
	-\frac{b^2}{2}|\xi|^2 & 0  & -\frac{b^2}{2}|\xi|^2-ib|\xi| & 0 \\
	0 & -\frac{a^2}{2}|\xi|^2 & 0  & -\frac{a^2}{2}|\xi|^2-ia|\xi|\\
	\end{array}}
\right).
\end{split}
\end{equation}
Moreover, the eigenvalue problem corresponding to \eqref{weshouldKelvinVoigt 2D} is
\begin{equation}\label{Se3 eigenvalue problem 2D}
\lambda\phi+\left(\frac{1}{2}|\xi|^{2}{B}_0-i|\xi|{B}_1\right)\phi=0,
\end{equation}
where $\lambda\in\mb{C}$ and $\phi\in\mb{C}^4$. The eigenvalue $\lambda=\lambda(|\xi|)$ of the problem \eqref{weshouldKelvinVoigt 2D} is the value of $\lambda$ satisfying \eqref{Se3 eigenvalue problem 2D} for $\phi\neq0$.

We now derive  energy estimates basing on the pointwise estimates of the partial Fourier transform of solutions in the Fourier space. In the pioneering paper \cite{UmedaKawashimaShizuta1984} the authors derived pointwise estimates for a general class of symmetric hyperbolic-parabolic systems by using the energy method in the Fourier space. Thus, we may apply Lemma 2.2 in \cite{UmedaKawashimaShizuta1984} to complete the following lemma.
\begin{lem}\label{Pointwise KV 2D}
	The solution $W=W(t,\xi)$ to the Cauchy problem \eqref{weshouldKelvinVoigt 2D} satisfies the following pointwise estimates for any $\xi\in\mb{R}^2$ and $t\geq0$:
	\begin{equation}\label{point wise KV estimates 2D}
	|W(t,\xi)|\lesssim e^{-c\rho(|\xi|)t}|W_0(\xi)|,
	\end{equation}
	where $\rho(|\xi|):=\frac{|\xi|^2}{1+|\xi|^2}$ and $c$ is a positive constant.
\end{lem}
\begin{rem}
	From the asymptotic behavior of eigenvalues $\lambda_j(|\xi|)$ for $j=1,\dots,4,$ which will be shown later in \eqref{asym KV 2D} and \eqref{asym KV large 2D}, the dissipative structure of the system \eqref{weshouldKelvinVoigt 2D} can be characterized by the property
	\begin{equation*}
	\emph{Re}\lambda_j(|\xi|)\leq -c\rho(|\xi|).
	\end{equation*}
	Moreover, according to the asymptotic expansions of eigenvalues for $|\xi|\rightarrow0$ and $|\xi|\rightarrow\infty$, our pointwise estimates of the partial Fourier transform of solutions stated in Lemma \ref{Pointwise KV 2D} are sharp.
\end{rem}
\begin{proof}
	%	Let us rewritten the system \eqref{weshouldKelvinVoigt 2D} in the following explicit form:
	%	\begin{equation}\label{explicit form KV 2D}
	%	\left\{
	%	\begin{aligned}
	%	&W_t^{(1)}+\left(\frac{1}{2}b^2|\xi|^2-ib|\xi|\right)W^{(1)}+\frac{1}{2}b^2|\xi|^2W^{(3)}=0,\\
	%	&W_t^{(2)}+\left(\frac{1}{2}a^2|\xi|^2-ia|\xi|\right)W^{(2)}+\frac{1}{2}a^2|\xi|^2W^{(4)}=0,\\
	%	&W_t^{(3)}+\frac{1}{2}b^2|\xi|^2W^{(1)}+\left(\frac{1}{2}b^2|\xi|^2+ib|\xi|\right)W^{(3)}=0,\\
	%	&W_t^{(4)}+\frac{1}{2}a^2|\xi|^2W^{(2)}+\left(\frac{1}{2}a^2|\xi|^2+ia|\xi|\right)W^{(4)}=0.
	%	\end{aligned}
	%	\right.
	%	\end{equation}
	%	{\color{red}
	%Let us define $U=\big(W^{(1)},W^{(3)}\big)^{\mathrm{T}}$. Thus, we obtain 
	%\begin{align}\label{modi 1}
	%U_t+\frac{1}{2}b^2|\xi|^2G_1U+ib|\xi|G_2U=0,
	%\end{align}	
	Let us define a real and antisymmetric matrix $\widetilde{K}$ by
	\[
	\begin{split}
	\widetilde{K}:=\left(
	{\begin{array}{*{20}c}
		0 & 0 & -\frac{b}{2} & 0\\
		0 & 0 & 0 & -\frac{a}{2}\\
		\frac{b}{2} & 0 & 0& 0\\
		0 & \frac{a}{2} & 0 &0 \\
		\end{array}}
	\right).
	\end{split}
	\]
	The coefficient matrices in \eqref{weshouldKelvinVoigt 2D}, i.e.,  $\frac{1}{2}B_0$ and $-B_1$ are real symmetric. Moreover, the matrix $\frac{1}{2}B_0$ is positive semi-definite. According to Lemma 2.2 in \cite{UmedaKawashimaShizuta1984}, due to the fact that $-\widetilde{K}B_1+\frac{1}{2}B_0$ is positive definite, we can derive
	\begin{align*}
	|W(t,\xi)|\lesssim e^{-\frac{c|\xi|^2}{1+|\xi|^2}t}|W_0(\xi)|,
	\end{align*}
	for all $t\geq0$ and $\xi\in\mb{R}^2$, with the positive constant $c$.
	%We find that $\frac{1}{2}b^2G_1$ is real symmetric and semi-positive definite, $bG_2$ is real symmetric. Moreover, there exists a real, anti-symmetric matrix $K_0$ defined by
	%\[
	%\begin{split}
	%{K}_0=\left(
	%{\begin{array}{*{20}c}
	%	0 & -\frac{b}{2}\\
	%	\frac{b}{2} & 0	\end{array}}
	%\right),
	%\end{split}
	%\]
	%such that $K_0G_1+G_2$ is positive definite. Thus, we directly apply Lemma 2.2 in \cite{UmedaKawashimaShizuta1984} to derive
	%\begin{align*}
	%|U(t,\xi)|^2\lesssim e^{-c\frac{|\xi|^2}{1+|\xi|^2}}|U_0(\xi)|^2,
	%\end{align*}
	%with positive constant $c$. Similarly, defining $V=\big(W^{(2)},W^{(4)}\big)^{\mathrm{T}}$ and following the same approach, we complete the proof.
	%}
\end{proof}

We concentrate on the estimates of the classical energy and higher-order energy of solutions with initial data taking from $H^{s}\big(\mb{R}^2\big)\cap L^m\big(\mb{R}^2\big)$. Because the proof is quite standard, we only sketch it.
\begin{thm}\label{KV estimate for W}
	Let us consider the Cauchy problem \eqref{linearproblem Kelvin-Voigt 2D} with initial data  satisfying $\left(|D|u^{(k)}_{0},u^{(k)}_{1}\right)\in \left(H^{s}\big(\mb{R}^2\big)\cap L^m\big(\mb{R}^2\big)\right)\times \left(H^{s}\big(\mb{R}^2\big)\cap L^m\big(\mb{R}^2\big)\right)$ for $k=1,2$, where $s\geq0$ and  $m\in[1,2]$. Then, we have the following estimates for the energies of higher-order:
	\begin{equation*}
	\begin{split}
	&\left\||D|^{s+1}u^{(k)}(t,\cdot)\right\|_{L^2(\mb{R}^2)}+\left\||D|^su^{(k)}_t(t,\cdot)\right\|_{L^2(\mb{R}^2)}\\
	&\qquad\qquad\lesssim(1+t)^{-\frac{2-m+ms}{2m}}\sum\limits_{k=1}^2\left\|\left(|D|u^{(k)}_{0},u^{(k)}_1\right)\right\|_{(H^{s}(\mb{R}^2)\cap L^m(\mb{R}^2))\times (H^{s}(\mb{R}^2)\cap L^m(\mb{R}^2))}.
	\end{split}
	\end{equation*}
\end{thm}
\begin{proof}
	We can derive the following estimates by the Parseval-Plancherel theorem and the pointwise estimates of the partial Fourier transform of solutions in the Fourier space in Lemma \ref{Pointwise KV 2D}:
	\begin{equation}
	\begin{split}
	\left\||D|^s\ml{F}_{\xi\rightarrow x}^{-1}(W)(t,\cdot)\right\|_{L^2(\mb{R}^2)}&=\||\xi|^sW(t,\xi)\|_{L^2(\mb{R}^2)}\lesssim\left\||\xi|^se^{-c\rho(|\xi|)t}W_0(\xi)\right\|_{L^2(\mb{R}^2)}\\
	&\lesssim\left\|\chi_{\intt}(\xi)|\xi|^se^{-c|\xi|^2t}W_0(\xi)\right\|_{L^2(\mb{R}^2)}\\
	&\quad+\left\|(\chi_{\midd}(\xi)+\chi_{\extt}(\xi))|\xi|^se^{-ct}W_0(\xi)\right\|_{L^2(\mb{R}^2)}.
	\end{split}
	\end{equation}
	For small frequencies, we apply H\"older's inequality and the Hausdorff-Young inequality to get
	\begin{equation*}
	\left\|\chi_{\intt}(\xi)|\xi|^se^{-c|\xi|^2t}W_0(\xi)\right\|_{L^2(\mb{R}^2)}\lesssim(1+t)^{-\frac{2-m+ms}{2m}}\left\|\ml{F}^{-1}(W_0)\right\|_{L^m(\mb{R}^2)}.
	\end{equation*}
	For middle and large frequencies, we immediate obtain
	\begin{equation*}
	\left\|(\chi_{\midd}(\xi)+\chi_{\extt}(\xi))|\xi|^se^{-ct}W_0(\xi)\right\|_{L^2(\mb{R}^2)}\lesssim e^{-ct}\left\|\ml{F}^{-1}(W_0)\right\|_{H^s(\mb{R}^2)}.
	\end{equation*}
	Together them completes the proof.
\end{proof}
\begin{rem}
	To derive the estimate of the solution itself to \eqref{linearproblem Kelvin-Voigt 2D} with 
	\begin{equation*}
	\left(|D|u_0^{(k)},u_1^{(k)}\right)\in\left(L^2\big(\mb{R}^2\big)\cap L^m\big(\mb{R}^2\big)\right)\times \left(L^2\big(\mb{R}^2\big)\cap L^m\big(\mb{R}^2\big)\right),
	\end{equation*} we need to estimate $\left\|\chi_{\intt}(\xi)|\xi|^{-1}W(t,\xi)\right\|_{L^2(\mb{R}^2)}$. Then, using H\"older's inequality, we have to derive the estimate of the following term: 
	\begin{equation*}
	\left\|\chi_{\intt}(\xi)|\xi|^{-1}e^{-c|\xi|^2t}\right\|_{L^\frac{2m}{2-m}(\mb{R}^2)}.
	\end{equation*}
	Nevertheless, due to a strong influence of the singularity for $|\xi|\rightarrow+0$, i.e., non integrability, the following inequality does not hold for all $t\geq0$ and $m\in[1,2]$:
	\begin{equation*}
	\left\|\chi_{\intt}(\xi)|\xi|^{-1}e^{-c|\xi|^2t}\right\|_{L^\frac{2m}{2-m}(\mb{R}^2)}<\infty.
	\end{equation*}
\end{rem}
\begin{rem}
	It is not reasonable for us to compare the estimates in Theorem \ref{KV estimate for W} with those in \cite{WuChaiLi2017} due to different assumptions of the data spaces. In the recent paper \cite{WuChaiLi2017}, the authors proved estimates for the classical energy to \eqref{linearproblem Kelvin-Voigt 2D} with initial data taking from $\left(H^{1}\big(\mb{R}^2\big)\cap L^1\big(\mb{R}^2\big)\right)\times \left(L^2\big(\mb{R}^2\big)\cap L^1\big(\mb{R}^2\big)\right)$ by using the energy method in the Fourier space and the Haraux-Komornik inequality.
\end{rem}

Next, we will establish $L^p-L^q$ estimates on the conjugate line by applying Lemma \ref{Pointwise KV 2D}.
\begin{thm}\label{LpLqawayconjugateline KV 2D} Let us consider the Cauchy problem \eqref{linearproblem Kelvin-Voigt 2D} with initial data satisfying $\left(|D|u^{(k)}_{0},u^{(k)}_{1}\right)\in\ml{S}\big(\mb{R}^2\big)\times \ml{S}\big(\mb{R}^2\big)$ for $k=1,2$. Then, the following estimates hold for the derivatives of the solution:
	\begin{equation*}
	\begin{split}
	&\left\||D|^{s+1}u^{(k)}(t,\cdot)\right\|_{L^q(\mb{R}^2)}+\left\||D|^su_t^{(k)}(t,\cdot)\right\|_{L^q(\mb{R}^2)}\\
	&\qquad\qquad\qquad\lesssim(1+t)^{-\frac{1}{2}\left(s+2\left(\frac{1}{p}-\frac{1}{q}\right)\right)}\sum\limits_{k=1}^2\left\|\left(|D|u^{(k)}_{0},u^{(k)}_1\right)\right\|_{H^{N_{s,p}}_p(\mb{R}^2)\times H^{N_{s,p}}_p(\mb{R}^2)},
	\end{split}
	\end{equation*}
	where $s\geq0$, $1\leq p\leq2$, $\frac{1}{p}+\frac{1}{q}=1$ and $N_{s,p}>s+2\left(\frac{1}{p}-\frac{1}{q}\right)$.
\end{thm}
\begin{rem} If we are interested in the case $p\in(1,2]$, then we can choose $N_{s,p}=s+2\left(\frac{1}{p}-\frac{1}{q}\right)$.
\end{rem}
\begin{proof} The proof of $L^p-L^q$ decay estimates on the conjugate line is divided into three parts.
	
	\emph{Part 1:} $L^p-L^q$ decay estimates for small frequencies.\\
	According to the derived $L^2-L^2$ estimates in Theorem \ref{KV estimate for W} we have
	\begin{equation}\label{L2-L2 KV 2D}
	\left\|\chi_{\intt}(D)|D|^s\ml{F}^{-1}_{\xi\rightarrow x}(W)(t,\cdot)\right\|_{L^2(\mb{R}^2)}\lesssim(1+t)^{-\frac{s}{2}}\left\|\ml{F}^{-1}(W_0)\right\|_{L^2(\mb{R}^2)}.
	\end{equation}
	According to Lemma \ref{Pointwise KV 2D} we give a $L^1-L^{\infty}$ estimate by applying the Hausdorff-Young inequality
	\begin{equation}\label{L1-Linfty KV 2D}
	\begin{split}
	\left\|\chi_{\intt}(D)|D|^s\ml{F}^{-1}_{\xi\rightarrow x}(W)(t,\cdot)\right\|_{L^{\infty}(\mb{R}^2)}
	&\lesssim\left\|\chi_{\intt}(\xi)|\xi|^se^{-c|\xi|^2t}\right\|_{L^1(\mb{R}^2)}\|W_0\|_{L^{\infty}(\mb{R}^2)}\\
	&\lesssim\left\|\chi_{\intt}(\xi)|\xi|^se^{-c|\xi|^{2}t}\right\|_{L^1(\mb{R}^2)}\left\|\ml{F}^{-1}(W_0)\right\|_{L^{1}(\mb{R}^2)}\\
	&\lesssim\left(\int_0^{\varepsilon}r^{s+1}e^{-cr^{2}t}dr\right)\left\|\ml{F}^{-1}(W_0)\right\|_{L^{1}(\mb{R}^2)}\\
	&\lesssim(1+t)^{-\frac{s+2}{2}}\left\|\ml{F}^{-1}(W_0)\right\|_{L^{1}(\mb{R}^2)}.
	\end{split}
	\end{equation}
	Combining \eqref{L2-L2 KV 2D} with \eqref{L1-Linfty KV 2D} and applying the Riesz-Thorin interpolation theorem, we obtain
	\begin{equation}\label{Lp-Lq small freq KV 2D}
	\begin{split}
	&\left\|\chi_{\intt}(D)|D|^s\ml{F}^{-1}_{\xi\rightarrow x}(W)(t,\cdot)\right\|_{L^{q}(\mb{R}^2)}\lesssim(1+t)^{-\frac{1}{2}\left(s+2\left(\frac{1}{p}-\frac{1}{q}\right)\right)}\left\|\ml{F}^{-1}(W_0)\right\|_{L^{p}(\mb{R}^2)},
	\end{split}
	\end{equation}
	where $1\leq p\leq 2$ and $\frac{1}{p}+\frac{1}{q}=1$.
	
	\emph{Part 2:} $L^p-L^q$ decay estimates for large frequencies.\\
	For $L^2-L^2$ estimates, we have
	\begin{equation*}
	\left\|\chi_{\extt}(D)|D|^s\ml{F}^{-1}_{\xi\rightarrow x}(W)(t,\cdot)\right\|_{L^2(\mb{R}^2)}\lesssim\left\|\langle D\rangle^s\ml{F}^{-1}(W_0)\right\|_{L^2(\mb{R}^2)}.
	\end{equation*}
	We now derive $L^1-L^{\infty}$ estimates. For $t\in(0,1]$, the following estimate holds:
	\begin{equation}\label{small time L1-Linfty KV 2D}
	\begin{split}
	\left\|\chi_{\extt}(D)|D|^s\ml{F}^{-1}_{\xi\rightarrow x}(W)(t,\cdot)\right\|_{L^{\infty}(\mb{R}^2)}&\lesssim\left\|\langle\xi\rangle^{-(2+\epsilon)}\right\|_{L^1(\mb{R}^2)}\left\|\langle\xi\rangle^{s+2+\epsilon}W_0\right\|_{L^{\infty}(\mb{R}^2)}\\
	&\lesssim\left\|\langle D\rangle^{s+2+\epsilon}\ml{F}^{-1}(W_0)\right\|_{L^1(\mb{R}^2)},
	\end{split}
	\end{equation}
	where $\epsilon$ is an arbitrary small positive constant. For $t\in[1,\infty)$, we immediately obtain exponential decay. Thus, combining with \eqref{small time L1-Linfty KV 2D} we derive
	\begin{equation*}
	\left\|\chi_{\extt}(D)|D|^s\ml{F}^{-1}_{\xi\rightarrow x}(W)(t,\cdot)\right\|_{L^{\infty}(\mb{R}^2)}\lesssim e^{-ct}\left\|\langle D\rangle^{s+2+\epsilon}\ml{F}^{-1}(W_0)\right\|_{L^1(\mb{R}^2)}.
	\end{equation*}
	Using the interpolation theorem again, we obtain
	\begin{equation}\label{Lp-Lq large freq KV 2D}
	\left\|\chi_{\extt}(D)|D|^s\ml{F}^{-1}_{\xi\rightarrow x}(W)(t,\cdot)\right\|_{L^{q}(\mb{R}^2)}\lesssim e^{-c\left(\frac{1}{p}-\frac{1}{q}\right)t}\left\|\langle D\rangle^{s+(2+\epsilon)\left(\frac{1}{p}-\frac{1}{q}\right)}\ml{F}^{-1}(W_0)\right\|_{L^p(\mb{R}^2)},
	\end{equation}
	where $1\leq p\leq 2$ and $\frac{1}{p}+\frac{1}{q}=1$.
	
	\emph{Part 3:} $L^p-L^q$ decay estimates for middle frequencies.\\
	According to the pointwise estimate of the partial Fourier transform of solutions in Lemma \ref{Pointwise KV 2D} when $\varepsilon\leq|\xi|\leq\frac{1}{\varepsilon}$, we derive
	\begin{equation*}
	\begin{split}
	\left\|\chi_{\midd}(D)|D|^s\ml{F}_{\xi\rightarrow x}^{-1}(W)(t,\cdot)\right\|_{L^2(\mb{R}^2)}&\lesssim e^{-ct}\left\|\ml{F}^{-1}(W_0)\right\|_{L^2(\mb{R}^2)},\\
	\left\|\chi_{\midd}(D)|D|^s\ml{F}_{\xi\rightarrow x}^{-1}(W)(t,\cdot)\right\|_{L^{\infty}(\mb{R}^2)}&\lesssim e^{-ct}\left\| \ml{F}^{-1}(W_0)\right\|_{L^1(\mb{R}^2)}.
	\end{split}
	\end{equation*}
	Then, we have from the interpolation theorem
	\begin{equation}\label{Lp-Lq middle freq KV 2D}
	\left\|\chi_{\midd}(D)|D|^s\ml{F}^{-1}_{\xi\rightarrow x}(W)(t,\cdot)\right\|_{L^{q}(\mb{R}^2)}\lesssim e^{-ct}\left\|\ml{F}^{-1}(W_0)\right\|_{L^p(\mb{R}^2)},
	\end{equation}
	where $1\leq p\leq 2$ and $\frac{1}{p}+\frac{1}{q}=1$.
	
	Summarizing \eqref{Lp-Lq small freq KV 2D}, \eqref{Lp-Lq middle freq KV 2D} and \eqref{Lp-Lq large freq KV 2D}, we complete the proof.
\end{proof}

\section{Asymptotic expansions and asymptotic representation}\label{Section KV 2D spectral}
With the aim of determining whether the pointwise estimates in the Fourier space in Lemma \ref{Pointwise KV 2D} are sharp or not, we now investigate the asymptotic expansion of eigenvalue of \eqref{weshouldKelvinVoigt 2D} for $|\xi|\rightarrow0$ and $|\xi|\rightarrow \infty$. Moreover, to derive the asymptotic profiles of solutions, we need to get the asymptotic expressions of the propagator $e^{t\hat{\Phi}(|\xi|)}$ (see the definition in \eqref{Se3 eigenvalue equation 2D}). This method is strongly motivated from the paper \cite{IdeHaramotoKawashima}.
\subsection{Asymptotic expansions}\label{Section KV 2D asymptic expansion}
We denote by $\lambda_j=\lambda_j(|\xi|)$, $j=1,\dots,4$, the eigenvalues of matrix $\hat{\Phi}(|\xi|)$. Thus, these eigenvalues are the solutions to the following characteristic equation:
\begin{equation}\label{characteristic eq. 2D}
\begin{split}
F(\lambda)&=\text{det}\left(\lambda I-\hat{\Phi}(|\xi|)\right)\\
&=\left|
{\begin{array}{*{20}c}
	\lambda+\frac{b^2}{2}|\xi|^2-ib|\xi| & 0  & \frac{b^2}{2}|\xi|^2 & 0 \\
	0  & \lambda+\frac{a^2}{2}|\xi|^2-ia|\xi| &  0 & \frac{a^2}{2}|\xi|^2\\
	\frac{b^2}{2}|\xi|^2 & 0  & \lambda+\frac{b^2}{2}|\xi|^2+ib|\xi| & 0 \\
	0 & \frac{a^2}{2}|\xi|^2 & 0  & \lambda+\frac{a^2}{2}|\xi|^2+ia|\xi|\\
	\end{array}}
\right|\\
&=\lambda^4+\left(a^2+b^2\right)|\xi|^2\lambda^3+\left(\left(a^2+b^2\right)|\xi|^2+a^2b^2|\xi|^4\right)\lambda^2+2a^2b^2|\xi|^4\lambda+a^2b^2|\xi|^4,
\end{split}
\end{equation}
where $|\xi|$ is regarded as a parameter. We notice that
\begin{equation*}
\frac{d}{d\lambda}F(\lambda)=4\lambda^3+3\left(a^2+b^2\right)|\xi|^2\lambda^2+2\left(\left(a^2+b^2\right)|\xi|^2+a^2b^2|\xi|^4\right)\lambda+2a^2b^2|\xi|^4.
\end{equation*}
Comparing the polynomials $F(\lambda)$ and $\frac{d}{d\lambda}F(\lambda)$, we observe there are non-trivial common divisors at most for value of frequencies $|\xi|$ in a zero measure set. Thus, we have only simple roots of $F(\lambda)=0$ outside of this zero measure set.

Let $P_j(|\xi|)$ be the corresponding eigenprojections, which can be expressed as the way
\begin{equation}\label{eigenprojections}
P_j(|\xi|)=\prod\limits_{k\neq j}\frac{\hat{\Phi}(|\xi|)-\lambda_k(|\xi|)I}{\lambda_j(|\xi|)-\lambda_k(|\xi|)},
\end{equation}
where $I$ is a identity matrix of dimensions $4\times 4$.

In the next step we distinguish the asymptotic expansions of eigenvalues and their corresponding eigenprojections between two cases: $|\xi|\rightarrow0$ and $|\xi|\rightarrow\infty$. These expansions essentially determine the asymptotic behavior of solutions.

\subsubsection{Asymptotic expansions for $|\xi|\rightarrow 0$} 
\smallskip
We deduce that the eigenvalues $\lambda_j(|\xi|)$ and their corresponding eigenprojections  $P_j(|\xi|)$ have the following asymptotic expansions for $|\xi|\rightarrow 0$, respectively:
\begin{align}
\lambda_j(|\xi|&)=\lambda_j^{(0)}+\lambda_j^{(1)}|\xi|+\lambda_j^{(2)}|\xi|^2+\cdots,\label{small expan 2D}\\
P_j(|\xi|)&=P_j^{(0)}+P_j^{(1)}|\xi|+P_j^{(2)}|\xi|^2+\cdots,\label{small expan matrix 2D}
\end{align}
where $\lambda_j^{(k)}\in\mb{C}$, $P_j^{(k)}\in\mb{C}^{4\times 4}$ for all $k\in\mb{N}$.

Then, we substitute $\lambda=\lambda_j(|\xi|)$ chosen in \eqref{small expan 2D} into the characteristic equation \eqref{characteristic eq. 2D} and calculate the coefficients $\lambda_j^{(k)}$. After lengthy but straightforward calculations, the value of pairwise distinct coefficients are given by
\begin{equation*}
\left\{\begin{aligned}
&\lambda_1^{(0)}=\lambda_2^{(0)}=\lambda_3^{(0)}=\lambda_4^{(0)}=0,\\
&\lambda_1^{(1)}=ib,\,\,\,\,\lambda_2^{(1)}=-ib,\,\,\,\,\lambda_3^{(1)}=ia,\,\,\,\,\lambda_4^{(1)}=-ia,\\
&\lambda_1^{(2)}=\lambda_2^{(2)}=-\frac{b^2}{2},\,\,\,\,\lambda_3^{(2)}=\lambda_4^{(2)}=-\frac{a^2}{2}.\\
\end{aligned}\right.
\end{equation*}
Consequently, the eigenvalues have the asymptotic behaviors for $|\xi|\rightarrow 0$
\begin{equation}\label{asym KV 2D}
\begin{aligned}
\lambda_1(|\xi|)&=ib|\xi|-\frac{b^2}{2}|\xi|^2+\ml{O}\big(|\xi|^3\big),&&\lambda_3(|\xi|)=ia|\xi|-\frac{a^2}{2}|\xi|^2+\ml{O}\big(|\xi|^3\big),\\
\lambda_2(|\xi|)&=-ib|\xi|-\frac{b^2}{2}|\xi|^2+\ml{O}\big(|\xi|^3\big),&&\lambda_4(|\xi|)=-ia|\xi|-\frac{a^2}{2}|\xi|^2+\ml{O}\big(|\xi|^3\big).
\end{aligned}
\end{equation}
By the pairwise distinct eigenvalues given in \eqref{asym KV 2D} and the matrix $\hat{\Phi}(|\xi|)$ given in \eqref{Se3 eigenvalue equation 2D}, we employ \eqref{eigenprojections} to calculate $P_j^{(0)}$ that
\begin{equation}\label{2D small matrix}
\begin{aligned}
P_1^{(0)}&=\diag(1,0,0,0),\quad &P_3^{(0)}=\diag(0,1,0,0),\\
P_2^{(0)}&=\diag(0,0,1,0),\quad &P_4^{(0)}=\diag(0,0,0,1).
\end{aligned}
\end{equation}
Thus, we get $P_j(|\xi|)-P_j^{(0)}=\ml{O}(|\xi|)$ for $|\xi|\rightarrow0$.

\subsubsection{Asymptotic expansions for $|\xi|\rightarrow \infty$}
\smallskip
Similarly, the eigenvalues $\lambda_j(|\xi|)$ and their corresponding eigenprojections $P_j(|\xi|)$ have the following asymptotic expansions for $|\xi|\rightarrow\infty$, respectively:
\begin{align}
\lambda_j(|\xi|)&=\lambda_j^{(0)}|\xi|^2+\lambda_j^{(1)}|\xi|+\lambda_j^{(2)}+\lambda_j^{(3)}|\xi|^{-1}+\lambda_j^{(4)}|\xi|^{-2}+\cdots,\label{large expan 2D}\\
P_j(|\xi|)&=P_j^{(0)}|\xi|^2+P_j^{(1)}|\xi|+P_j^{(2)}+P_j^{(3)}|\xi|^{-1}+P_j^{(4)}|\xi|^{-2}+\cdots,\label{large matrix expan 2D}
\end{align}
where $\lambda_j^{(k)}\in\mb{C}$ and $P_j^{(k)}\in\mb{C}^{4\times 4}$ for all $k\in\mb{N}$. 

We plug $\lambda=\lambda_j(|\xi|)$ chosen in \eqref{large expan 2D} into the characteristic equation \eqref{characteristic eq. 2D} and to obtain the pairwise distinct value of coefficients $\lambda_j^{(k)}$ 
\begin{equation*}
\left\{
\begin{aligned}
&\lambda_1^{(0)}=\lambda_2^{(0)}=0,\,\,\,\,\lambda_3^{(0)}=-b^2,\,\,\,\,\lambda_4^{(0)}=-a^2,\\
&\lambda_1^{(1)}=\lambda_2^{(1)}=\lambda_3^{(1)}=\lambda_4^{(1)}=0,\\
&\lambda_1^{(2)}=\lambda_2^{(2)}=-1,\,\,\,\,\lambda_3^{(2)}=\lambda_4^{(2)}=1,\\
&\lambda_1^{(3)}=\lambda_2^{(3)}=\lambda_3^{(3)}=\lambda_4^{(3)}=0,\\
&\lambda_1^{(4)}=-\frac{1}{b^2},\,\,\,\,\lambda_2^{(4)}=-\frac{1}{a^2},\,\,\,\,\lambda_3^{(4)}=\frac{1}{b^2},\,\,\,\,\lambda_4^{(4)}=\frac{1}{a^2}.
\end{aligned}\right.
\end{equation*}
Consequently, we investigate the following asymptotic behaviors of pairwise distinct eigenvalues $\lambda_j(|\xi|)$, $j=1,\dots,4$ for $|\xi|\rightarrow \infty$:
\begin{equation}\label{asym KV large 2D}
\begin{aligned}
\lambda_1(|\xi|)&=-1-\frac{1}{b^2}|\xi|^{-2}+\ml{O}\big(|\xi|^{-3}\big),&&\lambda_3(|\xi|)=-b^2|\xi|^2+1+\frac{1}{b^2}|\xi|^{-2}+\ml{O}\big(|\xi|^{-3}\big),\\
\lambda_2(|\xi|)&=-1-\frac{1}{a^2}|\xi|^{-2}+\ml{O}\big(|\xi|^{-3}\big),&&\lambda_4(|\xi|)=-a^2|\xi|^2+1+\frac{1}{a^2}|\xi|^{-2}+\ml{O}\big(|\xi|^{-3}\big).
\end{aligned}
\end{equation}
Also, by straightforward computations, we find that
\begin{equation}\label{2D large matrix}
\begin{split}
&P_j^{(0)}=P_j^{(1)}=\mathbf{0}_{4\times4}\,\,\,\,\text{for}\,\,\,\,j=1,2,3,4,\\
&P_1^{(2)}=\frac{1}{2}\left(
{\begin{array}{*{20}c}
	1 & 0  & -1 & 0 \\
	0  &0 &  0 & 0\\
	-1 & 0  & 1 & 0 \\
	0 & 0 & 0  & 0\\
	\end{array}}\right),\,\,\,\,P_2^{(2)}=\frac{1}{2}\left(
{\begin{array}{*{20}c}
	0 & 0  & 0 & 0 \\
	0  &1 &  0 & -1\\
	0 & 0  & 0 & 0 \\
	0 & -1 & 0  & 1\\
	\end{array}}\right),\\
&P_3^{(2)}=\frac{1}{2}\left(
{\begin{array}{*{20}c}
	1 & 0  & 1 & 0 \\
	0  &0 &  0 & 0\\
	1 & 0  & 1 & 0 \\
	0 & 0 & 0  & 0\\
	\end{array}}
\right),\,\,\,\,P_4^{(2)}=\frac{1}{2}\left(
{\begin{array}{*{20}c}
	0 & 0  & 0 & 0 \\
	0 & 1 &  0 & 1\\
	0 & 0  & 0 & 0 \\
	0 & 1 & 0  & 1\\
	\end{array}}
\right).
\end{split}
\end{equation}
It leads to $P_j(|\xi|)-P_j^{(0)}|\xi|^2-P_j^{(1)}|\xi|-P_j^{(2)}=\ml{O}\big(|\xi|^{-1}\big)$ for $|\xi|\rightarrow \infty$.

\subsection{Asymptotic representation}\label{Section KV 2D asymptic expresssion}
In the last subsection we have calculated the asymptotic expansions of eigenvalues $\lambda_j(|\xi|)$ and their corresponding eigenprojections $P_j(|\xi|)$ for $|\xi|\rightarrow0$ and $|\xi|\rightarrow\infty$. In order to give the representation of solution $W(t,\xi)=e^{t\hat{\Phi}(|\xi|)}W_0(\xi)$, motivated by \cite{IdeHaramotoKawashima}, in this part we will give asymptotic expressions of the propagator $e^{t\hat{\Phi}(|\xi|)}$ for $|\xi|\rightarrow0$ and $|\xi|\rightarrow\infty$.

The propagator $e^{t\hat{\Phi}(|\xi|)}$ has the following spectral decomposition:
\begin{equation}\label{2D small semigroup}
e^{t\hat{\Phi}(|\xi|)}=\sum\limits_{j=1}^4e^{\lambda_j(|\xi|)t}P_j(|\xi|),
\end{equation}
where $\lambda_j(|\xi|)$ are the eigenvalues of $\hat{\Phi}(|\xi|)$ and $P_j(|\xi|)$ are their corresponding eigenprojections.\\
Now, we distinguish between two cases: $|\xi|\rightarrow0$ and $|\xi|\rightarrow\infty$, respectively, to discuss the asymptotic expressions of \eqref{2D small semigroup}. 

\subsubsection{Asymptotic expansion of $e^{t\hat{\Phi}(|\xi|)}$ for $|\xi|\rightarrow 0$}
\smallskip
We first define the matrix $\hat{S}_0(t,\xi)$ by
\begin{equation}\label{S0 small 2D}
\hat{S}_0(t,\xi):=\sum\limits_{j=1}^4e^{\lambda_j^0(|\xi|)t}P_j^{(0)},
\end{equation}
where
\begin{equation}\label{lambda0 2D}
\begin{aligned}
\lambda_1^0(|\xi|)&=ib|\xi|-\frac{b^2}{2}|\xi|^2,&&\lambda_3^0(|\xi|)=ia|\xi|-\frac{a^2}{2}|\xi|^2,\\
\lambda_2^0(|\xi|)&=-ib|\xi|-\frac{b^2}{2}|\xi|^2,&&\lambda_4^0(|\xi|)=-ia|\xi|-\frac{a^2}{2}|\xi|^2,
\end{aligned}
\end{equation}
and $P_j^{(0)}$ are given in \eqref{2D small matrix}. We then rewrite the propagator for $|\xi|\rightarrow 0$ by
\begin{equation}\label{S0 and R0}
e^{t\hat{\Phi}(|\xi|)}=\hat{S}_0(t,\xi)+\hat{R}_0(t,\xi).
\end{equation}
Let us derive the pointwise estimate for the remainder $\hat{R}_0(t,\xi)$ for $|\xi|\rightarrow 0$.
\begin{lem}\label{small freq matrix KV 2D}
	We have the following estimate of remainder:
	\begin{equation}\label{lem 2D small}
	\left|\hat{R}_0(t,\xi)\right|\lesssim |\xi|e^{-c|\xi|^2t},
	\end{equation}
	where $0<|\xi|\leq \varepsilon$ and $c$ is a positive constant.
\end{lem}
\begin{proof}
	From \eqref{2D small semigroup}, \eqref{S0 small 2D} and \eqref{S0 and R0}, we can rewrite the remainder by
	\begin{equation*}
	\begin{split}
	\hat{R}_0(t,\xi)&=\sum\limits_{j=1}^4e^{\lambda_j(|\xi|)t}P_j(|\xi|)-\sum\limits_{j=1}^4e^{\lambda_j^0(|\xi|)t}P_j^{(0)}\\
	&=\sum\limits_{j=1}^4e^{\lambda_j(|\xi|)t}\left(P_j(|\xi|)-P_j^{(0)}\right)+\sum\limits_{j=1}^4e^{\lambda_j^0(|\xi|)t}\left(e^{\lambda_j(|\xi|)t-\lambda_j^0(|\xi|)t}-1\right)P_j^{(0)}\\
	&=:\hat{R}_{0,1}(t,\xi)+\hat{R}_{0,2}(t,\xi).
	\end{split}
	\end{equation*}
	Due to the facts that $P_j(|\xi|)-P_j^{(0)}=\ml{O}(|\xi|)$ and $e^{\lambda_j(|\xi|)t}\lesssim e^{-c|\xi|^2t}$ for $|\xi|\rightarrow 0$, we immediately obtain the estimate
	\begin{equation*}
	\left|\hat{R}_{0,1}(t,\xi)\right|\lesssim|\xi|e^{-c|\xi|^2t}\,\,\,\,\text{for}\,\,\,\,|\xi|\rightarrow0.
	\end{equation*}
	By the similar way, because $\lambda_j(|\xi|)-\lambda_j^0(|\xi|)=\ml{O}\big(|\xi|^{3}\big)$ and
	\begin{equation*}
	e^{c|\xi|^3t}-1=ct|\xi|^3\int_0^1e^{c|\xi|^3t\tau}d\tau,
	\end{equation*}
	we have
	\begin{equation*}
	\left|\hat{R}_{0,2}(t,\xi)\right|\lesssim \left|e^{\lambda_j(|\xi|)t-\lambda_j^0(|\xi|)t}-1\right|e^{-c|\xi|^2t}\lesssim t|\xi|^3e^{-c|\xi|^2t}\lesssim |\xi|e^{-c|\xi|^2t}\,\,\,\,\text{for}\,\,\,\,|\xi|\rightarrow0.
	\end{equation*}
	Summarizing above estimates, we complete the proof.
\end{proof}
\subsubsection{Asymptotic expansion of $e^{t\hat{\Phi}(|\xi|)}$ for $|\xi|\rightarrow \infty$}
\smallskip
From \eqref{large expan 2D} and \eqref{asym KV large 2D}, we now define
\begin{equation}\label{S0 large 2D}
\hat{S}_{\infty}(t,\xi):=\sum\limits_{j=1}^4e^{\lambda_j^{\infty}(|\xi|)t}\left(P_j^{(0)}|\xi|^2+P_j^{(1)}|\xi|+P_j^{(2)}\right),
\end{equation}
where
\begin{equation}\label{lambda large 2D}
\begin{aligned}
\lambda_1^{\infty}(|\xi|)&=-1-\frac{1}{b^2}|\xi|^{-2},&&\lambda_3^{\infty}(|\xi|)=-b^2|\xi|^2+1+\frac{1}{b^2}|\xi|^{-2},\\
\lambda_2^{\infty}(|\xi|)&=-1-\frac{1}{a^2}|\xi|^{-2},&&\lambda_4^{\infty}(|\xi|)=-a^2|\xi|^2+1+\frac{1}{a^2}|\xi|^{-2},
\end{aligned}
\end{equation}
and $P_j^{(0)},\,P_j^{(1)},\,P_j^{(2)}$ are given in \eqref{2D large matrix}. We write the propagator for $|\xi|\rightarrow\infty$
\begin{equation}\label{S0 and R0 large}
e^{t\hat{\Phi}(|\xi|)}=\hat{S}_{\infty}(t,\xi)+\hat{R}_{\infty}(t,\xi).
\end{equation}
Then, we can derive the following pointwise estimate for the remainder $\hat{R}_{\infty}(t,\xi)$ for $|\xi|\rightarrow\infty$.
\begin{lem}
	We have the following estimate of remainder:
	\begin{equation}\label{lem 2D large}
	\left|\hat{R}_{\infty}(t,\xi)\right|\lesssim e^{-ct},
	\end{equation}
	where $|\xi|\geq \frac{1}{\varepsilon}$ and $c$ is a positive constant.
\end{lem}
\begin{proof}
	From \eqref{S0 large 2D} and \eqref{S0 and R0 large}, we obtain
	\begin{equation*}
	\begin{split}
	\hat{R}_{\infty}(t,\xi)&=\sum\limits_{j=1}^4e^{\lambda_j(|\xi|)t}P_j(|\xi|)-\sum\limits_{j=1}^4e^{\lambda_j^{\infty}(|\xi|)t}\left(P_j^{(0)}|\xi|^2+P_j^{(1)}|\xi|+P_j^{(2)}\right)\\
	&=\sum\limits_{j=1}^4e^{\lambda_j(|\xi|)t}\left(P_j(|\xi|)-P_j^{(2)}\right)+\sum\limits_{j=1}^4e^{\lambda_j^{\infty}(|\xi|)t}\left(e^{\lambda_j(|\xi|)t-\lambda_j^{\infty}(|\xi|)t}-1\right)P_j^{(2)}\\
	&=:\hat{R}_{\infty,1}(t,\xi)+\hat{R}_{\infty,2}(t,\xi).
	\end{split}
	\end{equation*}
	Using the facts that $P_j(|\xi|)-P_j^{(2)}=\ml{O}\big(|\xi|^{-1}\big)$ and $e^{\lambda_j(|\xi|)t}\lesssim e^{-ct}$ for $|\xi|\rightarrow \infty$, we immediately obtain the estimate
	\begin{equation*}
	\left|\hat{R}_{\infty,1}(t,\xi)\right|\lesssim|\xi|^{-1} e^{-ct}\lesssim e^{-ct}\,\,\,\,\text{for}\,\,\,\,|\xi|\rightarrow\infty.
	\end{equation*}
	Since $\lambda_j(|\xi|)-\lambda_j^{\infty}(|\xi|)=\ml{O}\big(|\xi|^{-3}\big)$ we have
	\begin{equation*}
	\left|\hat{R}_{\infty,2}(t,\xi)\right|\lesssim \left|e^{\lambda_j(|\xi|)t-\lambda_j^{\infty}(|\xi|)t}-1\right|e^{-ct}\lesssim e^{-ct}\,\,\,\,\text{for}\,\,\,\,|\xi|\rightarrow\infty.
	\end{equation*}
	Thus, the proof is completed.
\end{proof}
\subsubsection{Conclusion for asymptotic representation}
\smallskip
Finally, combining of the estimate of $\hat{R}_0(t,\xi)$ and $\hat{R}_{\infty}(t,\xi)$, we can immediately prove the following statement for the asymptotic expansions of the propagator $e^{t\hat{\Phi}(|\xi|)}$. The proof of the next theorem strictly follows the proof of Lemma 4.3 from the paper \cite{IdeHaramotoKawashima}. 
\begin{thm}
	We have the following asymptotic expansions:
	\begin{equation}\label{representation of solution 2D}
	e^{t\hat{\Phi}(|\xi|)}=\hat{S}_0(t,\xi)+\hat{S}_{\infty}(t,\xi)+\hat{R}(t,\xi),
	\end{equation}
	where the remainder $\hat{R}(t,\xi)$ satisfies the estimates
	\begin{equation*}
	\left|\hat{R}(t,\xi)\right|\lesssim\left\{
	\begin{aligned}
	&|\xi|e^{-c|\xi|^2t}&&\text{for}\,\,\,\,|\xi|\leq \varepsilon,\\
	&e^{-ct}&&\text{for}\,\,\,\,|\xi|\geq \varepsilon.\\
	\end{aligned}
	\right.
	\end{equation*}
\end{thm}

\section{Asymptotic profiles}\label{Section KV 2D refinement of decay estimates}
The purpose of this section is to investigate the asymptotic profiles of solutions to \eqref{weshouldKelvinVoigt 2D}. According to Theorems \ref{KV estimate for W} and \ref{LpLqawayconjugateline KV 2D} in the previous section, we observe that the decay rate of these estimates is dominated by the behavior of the eigenvalues for $|\xi|\rightarrow0$. For frequencies in the bounded and large zones, they imply exponential decay providing that we assume a suitable regularity for initial data. Thus, we only explain the asymptotic profiles of solutions for the case $|\xi|\rightarrow0$.  

To begin with, let us introduce the following reference system:
\begin{equation}\label{reference system KV 2D}
\left\{\begin{aligned}
&\tilde{u}_{t}-\frac{1}{2}\widetilde{M}^2\Delta\tilde{u}+i\widetilde{M}(-\Delta)^{1/2}\tilde{u}=0,&&t>0,\,\,x\in\mb{R}^2,\\
&\tilde{u}(0,x)=\tilde{u}_0(x):=\ml{F}^{-1}(W_0)(x),&&x\in\mb{R}^2,
\end{aligned}\right.
\end{equation}
where $\tilde{u}=\left(\tilde{u}^{(1)},\tilde{u}^{(2)},\tilde{u}^{(3)},\tilde{u}^{(4)}\right)^{\mathrm{T}}$ and $\widetilde{M}=\diag(-b,-a,b,a)$.\\
This reference system is consisted of two different evolution equations as follows: 
\begin{equation*}
\begin{aligned}
&\text{heat equation:}&& \tilde{u}^+_t-\Delta\tilde{u}^+=0,\\
&\text{half-wave equation:}&& \tilde{u}^-_t\pm i(-\Delta)^{1/2}\tilde{u}^-=0,\\
\end{aligned}
\end{equation*}
with the suitable initial data.
\begin{rem}
	We point out that the influence of friction and Kelvin-Voigt damping to the asymptotic profiles of solution. Consider the linear dissipative elastic waves
	\begin{align}
	&u_{tt}-a^2\Delta u-\left(b^2-a^2\right)\nabla\divv u+u_t=0,\label{remark KV 2D friction}\\
	&u_{tt}-a^2\Delta u-\left(b^2-a^2\right)\nabla\divv u+\left(-a^2\Delta-\left(b^2-a^2\right)\nabla\divv\right)u_t=0,\label{remark KV 2D viscoelastic}
	\end{align}
	with $b^2>a^2>0$.	From the recent papers \cite{Reissig2016,ChenReissig2018}, we observe that the reference systems to \eqref{remark KV 2D friction} can be described by heat equations and heat equations with positive mass. However, when friction replaced by Kelvin-Voigt damping, i.e. dissipative system \eqref{remark KV 2D viscoelastic}, the reference system is consisted of heat equations and half-wave equations.
\end{rem}

Taking partial Fourier transform such that $\widetilde{W}(t,\xi)=\ml{F}_{x\rightarrow\xi}(\tilde{u}(t,x))$ to \eqref{reference system KV 2D}, we obtain
\begin{equation}\label{reference system KV 2D Fourier}
\left\{\begin{aligned}
&\widetilde{W}_t+\frac{1}{2}|\xi|^2\widetilde{M}^2\widetilde{W}+i|\xi|\widetilde{M}\widetilde{W}=0,&&t>0,\,\,\xi\in\mb{R}^2,\\
&\widetilde{W}(0,\xi)=W_0(\xi),&&\xi\in\mb{R}^2.
\end{aligned}\right.
\end{equation}
According to \eqref{S0 small 2D}, \eqref{lambda0 2D} and \eqref{2D small matrix}, we can explicitly express $\hat{S}_0(t,\xi)$ by
\begin{equation}\label{S_0 explicity KV 2D}
\begin{split}
\hat{S}_0(t,\xi)&=\diag\left(e^{\lambda_1^{(0)}(|\xi|)t},e^{\lambda_3^{(0)}(|\xi|)t},e^{\lambda_2^{(0)}(|\xi|)t},e^{\lambda_4^{(0)}(|\xi|)t}\right)\\
&=\diag\left(e^{ib|\xi|t-\frac{b^2}{2}|\xi|^2t},e^{ia|\xi|t-\frac{a^2}{2}|\xi|^2t},e^{-ib|\xi|t-\frac{b^2}{2}|\xi|^2t},e^{-ia|\xi|t-\frac{a^2}{2}|\xi|^2t}\right).
\end{split}
\end{equation}
Then, we know by direct computation that $\hat{S}_0(t,\xi)$ is the solution of the first-order system \eqref{reference system KV 2D Fourier}.

For one thing, considering initial data satisfying $\left(|D|u_0^{(k)},u_1^{(k)}\right)\in L^m\big(\mb{R}^2\big)\times L^m\big(\mb{R}^2\big)$ with $m\in[1,2]$ for $k=1,2$, we state our first result.
\begin{thm}\label{thme refinement 01}
	Let us consider the Cauchy problem \eqref{linearproblem Kelvin-Voigt 2D} with initial data satisfying $\left(|D|u^{(k)}_{0},u^{(k)}_{1}\right)\in L^m\big(\mb{R}^2\big)\times L^m\big(\mb{R}^2\big)$ for $k=1,2$.  Then, we obtain for the solution $W=W(t,\xi)$ to the Cauchy problem \eqref{weshouldKelvinVoigt 2D} the refinement estimate
	\begin{equation*}
	\left\|\chi_{\intt}(D)\ml{F}^{-1}_{\xi\rightarrow x}(W-\hat{S}_0W_0)(t,\cdot)\right\|_{\dot{H}^s(\mb{R}^2)}\lesssim(1+t)^{-\frac{2-m+ms}{2m}-\frac{1}{2}}\sum\limits_{k=1}^2\left\|\left(|D|u^{(k)}_{0},u^{(k)}_1\right)\right\|_{L^m(\mb{R}^2)\times L^m(\mb{R}^2)},
	\end{equation*}
	where $s\geq0$ and  $m\in[1,2]$.
\end{thm}
\begin{proof}
	According to \eqref{S0 and R0} and Lemma \ref{small freq matrix KV 2D} we can estimate
	\begin{equation}\label{refinement KV 2D 001}
	\begin{split}
	\left|\chi_{\intt}(\xi)|\xi|^s(W(t,\xi)-\hat{S}_0(t,\xi)W_0(\xi))\right|&=\left|\chi_{\intt}(\xi)|\xi|^s(e^{t\hat{\Phi}(|\xi|)}W_0(\xi)-\hat{S}_0(t,\xi)W_0(\xi))\right|\\
	&=\left|\chi_{\intt}(\xi)|\xi|^s\hat{R}_0(t,\xi)W_0(\xi)\right|\\
	&\lesssim\chi_{\intt}(\xi)|\xi|^{s+1}e^{-c|\xi|^2t}|W_0(\xi)|.
	\end{split}
	\end{equation}
	Then, applying the Parseval-Plancherel theorem we derive
	\begin{equation*}
	\left\|\chi_{\intt}(D)\ml{F}^{-1}_{\xi\rightarrow x}(W-\hat{S}_0W_0)(t,\cdot)\right\|_{\dot{H}^s(\mb{R}^2)}\lesssim\left\|\chi_{\intt}(\xi)|\xi|^{s+1}e^{-c|\xi|^2t}W_0(\xi)\right\|_{L^2(\mb{R}^2)}.
	\end{equation*}
	Following the procedure of the proof of Theorem \ref{KV estimate for W} we immediately complete this proof.
\end{proof}
For another, we consider the estimates basing on the $L^q$ norm for $q\in[2,\infty]$ in the next statement.
\begin{thm}\label{thme refinement 02}
	Let us consider the Cauchy problem \eqref{linearproblem Kelvin-Voigt 2D} with initial data satisfying $\left(|D|u_0^{(k)},u_1^{(k)}\right)\in\ml{S}\big(\mb{R}^2\big)\times\ml{S}\big(\mb{R}^2\big)$ for $k=1,2$. Then, we obtain for the solution $W=W(t,\xi)$ to the Cauchy problem \eqref{weshouldKelvinVoigt 2D} the refinement estimate
	\begin{equation*}
	\begin{split}
	&\left\|\chi_{\intt}(D)\ml{F}^{-1}_{\xi\rightarrow x}(W-\hat{S}_0W_0)(t,\cdot)\right\|_{\dot{H}^s_q(\mb{R}^2)}\\
	&\qquad\qquad\lesssim(1+t)^{-\frac{1}{2}\left(s+2\left(\frac{1}{p}-\frac{1}{q}\right)\right)-\frac{1}{2}}\sum\limits_{k=1}^2\left\|\left(|D|u_0^{(k)},u_1^{(k)}\right)\right\|_{L^p(\mb{R}^2)\times L^p(\mb{R}^2)},
	\end{split}
	\end{equation*}
	where $s\geq0$ and $1\leq p\leq 2$ and $\frac{1}{p}+\frac{1}{q}=1$.
\end{thm}
\begin{proof}
	Combining the proof of Theorem  \ref{LpLqawayconjugateline KV 2D} and  \eqref{refinement KV 2D 001}, the proof is complete.
\end{proof}
\begin{rem}
	Comparing Theorems \ref{KV estimate for W} and \ref{LpLqawayconjugateline KV 2D} with Theorems \ref{thme refinement 01} and \ref{thme refinement 02} we observe that by subtracting $\hat{S}_0(t,\xi)W_0(\xi)$ in the estimates, the decay rate $(1+t)^{-\frac{1}{2}}$ can be gained.
\end{rem}
Before we prove the asymptotic profiles with initial data belonging to weighted $L^1$ spaces, we recall the useful tool that Lemma 2.1 stated in \cite{Ikehata2004}.
\begin{lem}\label{Ikehata lemma L1gamma repeat}
	Let $\gamma\in[0,1]$ and $g\in L^{1,\gamma}(\mb{R}^n)$. Then, the following estimate holds:
	\begin{equation*}
	|\hat{g}(\xi)|\leq C_{\gamma}|\xi|^{\gamma}\|g\|_{L^{1,\gamma}(\mb{R}^n)}+\left|\int_{\mb{R}^n}g(x)dx\right|,
	\end{equation*}
	with some constant $C_{\gamma}>0$.
\end{lem}
\begin{thm}\label{thme refinement 03}
	Let us consider the Cauchy problem \eqref{linearproblem Kelvin-Voigt 2D} with initial data satisfying $\left(|D|u^{(k)}_{0},u^{(k)}_{1}\right)\in L^{1,\gamma}\big(\mb{R}^2\big)\times L^{1,\gamma}\big(\mb{R}^2\big)$ for $k=1,2$.  Then, we obtain for the solution $W=W(t,\xi)$ to the Cauchy problem \eqref{weshouldKelvinVoigt 2D} the refinement estimate
	\begin{equation*}
	\begin{split}
	\left\|\chi_{\intt}(D)\ml{F}^{-1}_{\xi\rightarrow x}(W-\hat{S}_0W_0)(t,\cdot)\right\|_{\dot{H}^s(\mb{R}^2)}&\lesssim(1+t)^{-\frac{1+s}{2}-\frac{1+\gamma}{2}}\sum\limits_{k=1}^2\left\|\left(|D|u^{(k)}_{0},u^{(k)}_1\right)\right\|_{L^{1,\gamma}(\mb{R}^2)\times L^{1,\gamma}(\mb{R}^2)}\\
	&\quad+(1+t)^{-\frac{1+s}{2}-\frac{1}{2}}\sum\limits_{k=1}^2\left|\int_{\mb{R}^2}\left(|D|u^{(k)}_0(x)+u^{(k)}_1(x)\right)dx\right|,
	\end{split}
	\end{equation*}
	where $s\geq0$ and  $\gamma\in[0,1]$.
\end{thm}
\begin{rem}
	Theorem \ref{thme refinement 03} shows that if we take initial data satisfying
	\begin{equation*}
	\left|\int_{\mb{R}^2}\left(|D|u^{(k)}_0(x)+u^{(k)}_1(x)\right)dx\right|=0\,\,\,\,\text{for}\,\,\,\,k=1,2,
	\end{equation*}
	then the decay rates given in Theorem \ref{thme refinement 01} when $m=1$ can be improved by $(1+t)^{-\frac{\gamma}{2}}$ for $\gamma\in(0,1]$.
\end{rem}
\begin{proof}
	From our derived estimate \eqref{refinement KV 2D 001}, we get
	\begin{equation*}
	\begin{split}
	&\left|\chi_{\intt}(\xi)|\xi|^s(W(t,\xi)-\hat{S}_0(t,\xi)W_0(\xi))\right|\\
	&\qquad\lesssim\chi_{\intt}(\xi)|\xi|^{s+1}e^{-c|\xi|^2t}|W_0(\xi)|\\
	&\qquad\lesssim\chi_{\intt}(\xi)|\xi|^{s+1+\gamma}e^{-c|\xi|^2t}\sum\limits_{k=1}^2\left\|\left(|D|u^{(k)}_{0},u^{(k)}_1\right)\right\|_{L^{1,\gamma}(\mb{R}^2)\times L^{1,\gamma}(\mb{R}^2)}\\
	&\qquad\quad+\chi_{\intt}(\xi)|\xi|^{s+1}e^{-c|\xi|^2t}\sum\limits_{k=1}^2\left|\int_{\mb{R}^2}\left(|D|u^{(k)}_0(x)+u^{(k)}_1(x)\right)dx\right|,
	\end{split}
	\end{equation*}
	where we used Lemma \ref{Ikehata lemma L1gamma repeat} in the last inequality.\\
	Next, using the Parseval-Plancherel theorem we immediately complete the proof.
\end{proof}

\section{Weakly coupled system of semilinear elastic waves with Kelvin-Voigt damping in 2D}\label{global existence small data solution KV 2D}
One of our goal in this paper is to develop sharp energy estimates of solutions to the linear elastic waves with Kelvin-Voigt damping in 2D with initial data \begin{equation*}
\left(|D|u_0^{(k)},u_1^{(k)}\right)\in\left(H^s\big(\mb{R}^2\big)\cap L^m\big(\mb{R}^2\big)\right)\times \left(H^s\big(\mb{R}^2\big)\cap L^m\big(\mb{R}^2\big)\right)
\end{equation*} for all $s\geq0$ and $m\in[1,2]$. If initial data is supposed to $\ml{A}_{m,s}\big(\mb{R}^2\big)$ for $s\geq0$ and $m\in[1,2]$, one can obtain the next theorem.
\begin{thm}\label{Thm KV 2D energy}
	Let us consider the Cauchy problem \eqref{linearproblem Kelvin-Voigt 2D} with initial data  satisfying $\left(u^{(k)}_{0},u^{(k)}_{1}\right)\in \ml{A}_{m,s}\big(\mb{R}^2\big)$ for $k=1,2$, where $s\geq0$ and  $m\in[1,2]$. Then, we have the following estimates:
	\begin{align}
	\left\|u^{(k)}(t,\cdot)\right\|_{L^2(\mb{R}^2)}&\lesssim(1+t)^{1-\frac{2-m}{2m}}\sum\limits_{k=1}^2\left\|\left(u^{(k)}_0,u^{(k)}_1\right)\right\|_{\ml{A}_{m,0}(\mb{R}^2)},\label{DATA Kelvin-Vogit 2D solution}\\
	\left\||D|u^{(k)}(t,\cdot)\right\|_{\dot{H}^s(\mb{R}^2)}+\left\|u_t^{(k)}(t,\cdot)\right\|_{\dot{H}^s(\mb{R}^2)}&\lesssim(1+t)^{-\frac{2-m+ms}{2m}}\sum\limits_{k=1}^2\left\|\left(u^{(k)}_0,u^{(k)}_1\right)\right\|_{\ml{A}_{m,s}(\mb{R}^2)}.\label{DATA Kelvin-Vogit 2D derivatives}
	\end{align}
\end{thm}
\begin{proof}
	To get the estimate \eqref{DATA Kelvin-Vogit 2D derivatives}, we recall the following estimate in the Fourier space from Lemma 2.4 in the recent papers \cite{IkehataNatsume2012,WuChaiLi2017}:
	\begin{equation*}
	|\xi|^2|\hat{u}(t,\xi)|^2+|\hat{u}_t(t,\xi)|^2\lesssim e^{-c\rho_{\varepsilon_1}(|\xi|)t}\left(|\xi|^2|\hat{u}_0(\xi)|^2+|\hat{u}_1(\xi)|^2\right),
	\end{equation*}
	where 
	\begin{equation*}
	\rho_{\varepsilon_1}(|\xi|):=\left\{
	\begin{aligned}
	&\varepsilon_1|\xi|^2&&\text{for}\,\,\,\,|\xi|\leq1,\\
	&\varepsilon_1&&\text{for}\,\,\,\,|\xi|>1,
	\end{aligned}
	\right.
	\end{equation*}
	with $\varepsilon_1>0$. Then, the application of H\"older's inequality and the Hausdorff-Young inequality immediately implies \eqref{DATA Kelvin-Vogit 2D derivatives}.
	
	To develop the estimate of the solution itself, we only need to combine the following integral formula:
	\begin{equation*}
	u^{(k)}(t,x)=\int_0^tu_{\tau}^{(k)}(\tau,x)d\tau+u_0^{(k)}(x)
	\end{equation*}
	and the estimate for $\left\|u_t^{(k)}(t,\cdot)\right\|_{L^2(\mb{R}^2)}$ in \eqref{DATA Kelvin-Vogit 2D derivatives} to complete \eqref{DATA Kelvin-Vogit 2D solution}.
\end{proof}

Let us consider the following weakly coupled system of semilinear elastic waves with Kelvin-Voigt damping in two dimensional space:
\begin{equation}\label{semilinear problem Kelvin-Voigt 2D}
\left\{
\begin{aligned}
&u_{tt}-a^2\Delta u-\left(b^2-a^2\right)\nabla\divv u+\left(-a^2\Delta -\left(b^2-a^2\right)\nabla \divv\right)u_t=f(u),&&t>0,\,\,x\in\mb{R}^2,\\
&(u,u_t)(0,x)=(u_0,u_1)(x),&&x\in\mb{R}^2,
\end{aligned}
\right.
\end{equation}
where $b^2>a^2>0$ and the nonlinear terms on the right-hand sides are
\begin{equation*}
f(u):=\left(|u^{(2)}|^{p_1},|u^{(1)}|^{p_2}\right)^{\mathrm{T}}\,\,\,\,\text{with}\,\,\,\,p_1,p_2>1.
\end{equation*}
Using estimates \eqref{DATA Kelvin-Vogit 2D solution}, \eqref{DATA Kelvin-Vogit 2D derivatives}, Duhamel's principle and some tools in Harmonic Analysis (e.g. the Gagliardo-Nirenberg inequality), one can prove the global (in time) existence of small data Sobolev solutions to \eqref{semilinear problem Kelvin-Voigt 2D}. 

Before stating our result for the global (in time) existence of small data energy solution, we introduce the balanced exponent $p_{\text{bal}}(m)$ by
\begin{equation}\label{balanced exponent 01}
p_{\text{bal}}(m):=\frac{2(m+2)}{2-m}\,\,\,\,\text{with}\,\,\,\,m\in[1,2),
\end{equation}
and the balanced parameters
\begin{equation}\label{balanced exponent 02}
\alpha_k(m):=\frac{2(1+m)+(3m+2)p_k+mp_1p_2}{2(p_1p_2-1)}\,\,\,\,\text{with}\,\,\,\,m\in[1,2)\,\,\,\,\text{for}\,\,\,\,k=1,2.
\end{equation}
\begin{rem}
	We observe the relation between the balanced exponent \eqref{balanced exponent 01} and balanced parameters \eqref{balanced exponent 02}. For one thing, if we consider the condition $\alpha_1(m)<1$, it also can be rewritten by
	\begin{equation*}
	p_1\left(p_2+1-p_{\text{bal}}(m)\right)>p_{\text{bal}}(m).
	\end{equation*}
	For another, if we consider the condition  $\alpha_2(m)<1$, it also can be rewritten by
	\begin{equation*}
	p_2\left(p_1+1-p_{\text{bal}}(m)\right)>p_{\text{bal}}(m).
	\end{equation*}
\end{rem}
\begin{thm}\label{global existence KV 2D}
	Let $b^2>a^2>0$ in \eqref{semilinear problem Kelvin-Voigt 2D}. Let us assume $p_1,p_2>1$, satisfying one of the following conditions:
	\begin{enumerate}
		\item we assume $p_{\text{bal}}(m)<\min\{p_1;p_2\}$;
		\item we assume $\alpha_1(m)<1$ if $\frac{2}{m}\leq p_2\leq p_{\text{bal}}(m)<p_1$;
		\item we assume $\alpha_2(m)<1$ if $\frac{2}{m}\leq p_1\leq p_{\text{bal}}(m)<p_2$.
	\end{enumerate}
	Then, there exists a constant $\varepsilon_0>0$ such that for any initial data $\left(u_0^{(k)},u_1^{(k)}\right)\in\ml{A}_{m,0}\big(\mb{R}^2\big)$ for $k=1,2,$ and $m\in[1,2)$ with
	\begin{equation*}
	\left\|\left(u_0^{(1)},u_1^{(1)}\right)\right\|_{\ml{A}_{m,0}(\mb{R}^2)}+\left\|\left(u_0^{(2)},u_1^{(2)}\right)\right\|_{\ml{A}_{m,0}(\mb{R}^2)}\leq\varepsilon_0,
	\end{equation*}
	there is uniquely determined energy solution
	\begin{equation*}
	u\in\left(\ml{C}\left([0,\infty),H^1\big(\mb{R}^2\big)\right)\cap\ml{C}^1\left([0,\infty),L^2\big(\mb{R}^2\big)\right)\right)^2
	\end{equation*}
	to \eqref{semilinear problem Kelvin-Voigt 2D}. Moreover, the solutions satisfies the following estimates:
	\begin{equation*}
	\begin{split}
	\left\|u^{(k)}(t,\cdot)\right\|_{L^2(\mb{R}^2)}&\lesssim(1+t)^{1-\frac{2-m}{2m}+\ell_k}\sum\limits_{k=1}^2\left\|\left(u^{(k)}_0,u^{(k)}_1\right)\right\|_{\ml{A}_{m,0}(\mb{R}^2)},\\
	\left\||D| u^{(k)}(t,\cdot)\right\|_{L^2(\mb{R}^2)}+\left\|u_t^{(k)}(t,\cdot)\right\|_{L^2(\mb{R}^2)}&\lesssim(1+t)^{-\frac{2-m}{2m}+\ell_k}\sum\limits_{k=1}^2\left\|\left(u^{(k)}_0,u^{(k)}_1\right)\right\|_{\ml{A}_{m,0}(\mb{R}^2)},
	\end{split}
	\end{equation*}
	where
	\begin{equation}\label{loss of decay}
	0\leq\ell_k=\ell_k(m,p_k):=\left\{
	\begin{aligned}
	&0&&\text{if}\,\,\,\,p_k>p_{\text{bal}}(m),\\
	&\epsilon_0&&\text{if}\,\,\,\,p_k=p_{\text{bal}}(m),\\
	&\tfrac{2-m}{2m}\left(p_{\text{bal}}(m)-p_k\right)&&\text{if}\,\,\,\,p_k<p_{\text{bal}}(m),
	\end{aligned}
	\right.
	\end{equation}
	represent the (no) loss of decay in comparison with the corresponding decay estimates for the solution to the linear Cauchy problem \eqref{linearproblem Kelvin-Voigt 2D} (see Theorem \ref{Thm KV 2D energy}), with $\epsilon_0>0$ being an arbitrary small constant in the limit cases that $p_k=p_{\text{bal}}(m)$ for $k=1,2$.
\end{thm}
\begin{rem}
	Let us recall the weakly coupled system of semilinear damped wave equations
	\begin{equation}\label{weakly coupled system of damped wave}
	\left\{
	\begin{aligned}
	&u_{tt}-\Delta u+u_t=|v|^{p_1},&&t>0,\,\,x\in\mb{R}^n,\\
	&v_{tt}-\Delta v+v_t=|u|^{p_2},&&t>0,\,\,x\in\mb{R}^n,\\
	&(u,u_t,v,v_t)(0,x)=(u_0,u_1,v_0,v_1)(x),&&x\in\mb{R}^n,
	\end{aligned}
	\right.
	\end{equation}
	with $n\in\mb{N}$ and $p_1,p_2>1$. The papers \cite{Narazaki2009,NishiharaWakasugi,SunWang2007} proved the condition for the existence of global (in time) Sobolev solutions to \eqref{weakly coupled system of damped wave}, which can be described by
	\begin{equation}\label{condition GESDS damped wave}
	\frac{1+\max\{p_1;p_2\}}{p_1p_2-1}<\frac{n}{2},\,\,\,\,\text{and especially in two dimensional case}\,\,\,\,\frac{1+\max\{p_1;p_2\}}{p_1p_2-1}<1.
	\end{equation}
	For the system \eqref{semilinear problem Kelvin-Voigt 2D}, we interpret the term $\left(-a^2\Delta-\left(b^2-a^2\right)\nabla\divv\right)u_t$ as a damping term for the elastic waves. Thus, we point out the condition for the global (in time) existence of small data energy solution to \eqref{semilinear problem Kelvin-Voigt 2D} is
	\begin{equation*}
	\alpha_{\max}(m)=\max\left\{\alpha_1(m);\alpha_2(m)\right\}=\frac{1+m+\frac{3m+2}{2}\max\{p_1;p_2\}+\frac{m}{2}p_1p_2}{p_1p_2-1}<1.
	\end{equation*}
\end{rem}
\begin{proof}
	First of all, by Duhamel's principle, we can reduce to consider the linear problem \eqref{linearproblem Kelvin-Voigt 2D}. In the following, we denote by $K_0=K_0(t,x)$ and $K_1=K_1(t,x)$ the fundamental solutions to the linear problem, corresponding to initial data, namely,
	\begin{equation*}
	u(t,x)=K_0(t,x)\ast_{(x)}u_0(x)+K_1(t,x)\ast_{(x)}u_1(x)
	\end{equation*}
	is the solution to \eqref{linearproblem Kelvin-Voigt 2D}.
	
	Let us define for $T>0$ the spaces of solutions $X(T)$ by
	\begin{equation*}
	X(T):=\left(\ml{C}\left([0,T],H^1\big(\mb{R}^2\big)\right)\cap \ml{C}^1\left([0,T],L^2\big(\mb{R}^2\big)\right)\right)^2
	\end{equation*}
	with the corresponding norm
	\begin{equation*}
	\|u\|_{X(T)}:=\sup\limits_{t\in[0,T]}\left((1+t)^{-\ell_1}M_1\left(t;u^{(1)}\right)+(1+t)^{-\ell_2}M_2\left(t;u^{(2)}\right)\right),
	\end{equation*}
	where
	\begin{align*}
	M_k\left(t;u^{(k)}\right):=&(1+t)^{-1+\frac{2-m}{2m}}\left\|u^{(k)}(t,\cdot)\right\|_{L^2(\mb{R}^2)}\\
	&+(1+t)^{\frac{2-m}{2m}}\left(\left\||D| u^{(k)}(t,\cdot)\right\|_{L^2(\mb{R}^2)}+\left\|u_t^{(k)}(t,\cdot)\right\|_{L^2(\mb{R}^2)}\right),
	\end{align*}
	and the parameters in the loss of decay ($\ell_k>0$) and no loss of decay ($\ell_k=0$) are defined in \eqref{loss of decay}.
	
	Next, we consider the integral operator $N:X(T)\rightarrow X(T)$, which is defined by 
	\begin{equation*}
	Nu(t,x):=u_{\text{lin}}(t,x)+u_{\text{non}}^{(1)}(t,x)+u_{\text{non}}^{(2)}(t,x),	
	\end{equation*}
	where 
	%	\begin{equation*}
	%	N_ku(t,x)=u^{(k)}_{\text{lin}}(t,x)+u^{(k)}_{\text{non}}(t,x)\,\,\,\,\text{for}\,\,\,\,k=1,2,
	%	\end{equation*}
	%	with
	\begin{equation*}
	\begin{split}
	u_{\text{lin}}(t,x)&=K_0(t,x)\ast_{(x)}u_0(x)+K_1(t,x)\ast_{(x)}u_1(x),\\
	u^{(1)}_{\text{non}}(t,x)&=\int_0^tK_1(t-\tau,x)\ast_{(x)}\left(|u^{(2)}(\tau,x)|^{p_1},0\right)^{\mathrm{T}}d\tau,\\
	u^{(2)}_{\text{non}}(t,x)&=\int_0^tK_1(t-\tau,x)\ast_{(x)}\left(0,|u^{(1)}(\tau,x)|^{p_2}\right)^{\mathrm{T}}d\tau.
	\end{split}
	\end{equation*}
	
	From Theorem \ref{Thm KV 2D energy} and $\ell_k\geq0$ for $k=1,2$, we can get the following estimates:
	\begin{equation*}
	\left\|u_{\text{lin}}\right\|_{X(T)}\lesssim\sum\limits_{k=1}^2\left\|\left(u_0^{(k)},u_1^{(k)}\right)\right\|_{\ml{A}_{m,0}(\mb{R}^2)}.
	\end{equation*}
	
	In the next step we should estimate these terms 
	\begin{equation*}
	\left\|\partial_t^j|D|^l u_{\text{non}}^{(1)}(t,\cdot)\right\|_{L^2(\mb{R}^2)}\,\,\,\,\text{and}\,\,\,\,\left\|\partial_t^j|D|^l u_{\text{non}}^{(2)}(t,\cdot)\right\|_{L^2(\mb{R}^2)}
	\end{equation*}
	for $j+l=0,1$ and $j,l\in\mb{N}$.\\
	Applying the classical Gagliardo-Nirenberg inequality, we may obtain
	\begin{equation*}
	\begin{split}
	\left\||u^{(2)}(\tau,\cdot)|^{p_1}\right\|_{L^m(\mb{R}^2)}&\lesssim(1+\tau)^{\left(-\frac{2-m}{2m}+\ell_2\right)p_1+\frac{2}{m}}\|u\|_{X(\tau)}^{p_1},\\
	\left\||u^{(1)}(\tau,\cdot)|^{p_2}\right\|_{L^m(\mb{R}^2)}&\lesssim(1+\tau)^{\left(-\frac{2-m}{2m}+\ell_1\right)p_2+\frac{2}{m}}\|u\|_{X(\tau)}^{p_2},\\
	\end{split}
	\end{equation*}
	where we use our assumption $\frac{2}{m}\leq\min\{p_1;p_2\}$ for $m\in[1,2)$.\\
	For one thing, in order to estimate of $u^{(k)}_{\text{non}}$ for $k=1,2,$ we apply the derived $\left(L^2\cap L^m\right)-L^2$ estimate in $[0,t]$. For another, we use the derived $\left(L^2\cap L^m\right)-L^2$ estimate in $[0,t/2]$ and the derive $L^2-L^2$ estimate in $[t/2,t]$ to estimate  $\partial_t^j|D|^lu^{(k)}_{\text{non}}$ for $j+l=1$ and $k=1,2$. Therefore, we obtain the following estimates for $j+l=0,1$ and $j,l\in\mb{N}$:
	\begin{equation*}
	\begin{split}
	&(1+t)^{j+l-1+\frac{2-m}{2m}-\ell_1}\left\|\partial_t^j|D|^lu^{(1)}_{\text{non}}(t,\cdot)\right\|_{L^2(\mb{R}^2)}\\
	&\qquad\qquad\lesssim(1+t)^{-\ell_1}\|u\|_{X(t)}^{p_1}\left(\int_0^{t/2}(1+\tau)^{\left(-\frac{2-m}{2m}+\ell_2\right)p_1+\frac{2}{m}}d\tau+(1+t)^{\left(-\frac{2-m}{2m}+\ell_2\right)p_1+\frac{2}{m}+1}\right),\\
	&(1+t)^{j+l-1+\frac{2-m}{2m}-\ell_2}\left\|\partial_t^j|D|^lu^{(2)}_{\text{non}}(t,\cdot)\right\|_{L^2(\mb{R}^2)}\\
	&\qquad\qquad\lesssim(1+t)^{-\ell_2}\|u\|_{X(t)}^{p_2}\left(\int_0^{t/2}(1+\tau)^{\left(-\frac{2-m}{2m}+\ell_1\right)p_2+\frac{2}{m}}d\tau+(1+t)^{\left(-\frac{2-m}{2m}+\ell_1\right)p_2+\frac{2}{m}+1}\right).
	\end{split}
	\end{equation*}
	We now need to distinguish between three cases. Without loss of generality, we only give the proof for the case $p_1>p_2$.
	
	\emph{Case 1:} We assume $p_{\text{bal}}(m)<\min\{p_1;p_2\}$.\\
	In this case it allows us to assume no loss of decay, i.e., $\ell_1=\ell_2=0$. Our assumption $p_{\text{bal}}(m)<\min\{p_1;p_2\}$ immediately leads to
	\begin{equation*}
	-\frac{2-m}{2m}p_1+\frac{2}{m}<-1\,\,\,\,\text{and}\,\,\,\,-\frac{2-m}{2m}p_2+\frac{2}{m}<-1.
	\end{equation*}
	Hence, we have estimates for $j+l=0,1$ and $j,l\in\mb{N}$
	\begin{equation*}
	\begin{split}
	(1+t)^{j+l-1+\frac{2-m}{2m}}\left(\left\|\partial_t^j|D|^lu^{(1)}_{\text{non}}(t,\cdot)\right\|_{L^2(\mb{R}^2)}+\left\|\partial_t^j|D|^lu^{(2)}_{\text{non}}(t,\cdot)\right\|_{L^2(\mb{R}^2)}\right)\lesssim\|u\|_{X(t)}^{p_1}+\|u\|_{X(t)}^{p_2}.
	\end{split}
	\end{equation*}
	
	\emph{Case 2:} We assume $\alpha_1(m)<1$ if $\frac{2}{m}\leq p_2\leq p_{\text{bal}}(m)<p_1$.\\
	In this case it allows us to assume loss of decay only for the second component and its derivatives with respect to $x$ and $t$, i.e., $\ell_1=0$ and
	\begin{equation}\label{decay rate KV 2D}
	\ell_2=\left\{
	\begin{aligned}
	&\epsilon_0&&\text{if}\,\,\,\,p_2=p_{\text{bal}}(m),\\
	&\tfrac{2-m}{2m}\left(p_{\text{bal}}(m)-p_2\right)&&\text{if}\,\,\,\,p_2<p_{\text{bal}}(m).
	\end{aligned}
	\right.
	\end{equation}
	Due to the assumption $\frac{2}{m}\leq p_2\leq p_{\text{bal}}(m)$, the parameter $\ell_2$ chosen in \eqref{decay rate KV 2D} is positive. Moreover, the assumption $\alpha_1(m)<1$ implies that
	\begin{equation}\label{equiv}
	1+\frac{2}{m}-\frac{2-m}{2m}p_1+\frac{2-m}{2m}\left(\frac{2(m+2)}{2-m}-p_2\right)p_1<0.
	\end{equation}
	We can get these estimates from the combination of the parameter $\ell_2$ chosen in \eqref{decay rate KV 2D}, our assumptions \eqref{equiv} and $\frac{2}{m}\leq p_2\leq p_{\text{bal}}(m)<p_1$
	\begin{equation*}
	-\frac{2-m}{2m}p_2+\frac{2}{m}+1-\ell_2\leq 0\,\,\,\,\text{and}\,\,\,\,\left(-\frac{2-m}{2m}+\ell_2\right)p_1+\frac{2}{m}+1<0.
	\end{equation*}
	Then, the following estimates hold for $j+l=0,1$ and $j,l\in\mb{N}$:
	\begin{equation*}
	\begin{split}
	(1+t)^{j+l-1+\frac{2-m}{2m}}\left\|\partial_t^j|D|^lu^{(1)}_{\text{non}}(t,\cdot)\right\|_{L^2(\mb{R}^2)}&\lesssim\|u\|_{X(t)}^{p_1},\\
	(1+t)^{j+l-1+\frac{2-m}{2m}-\ell_2}\left\|\partial_t^j|D|^lu^{(2)}_{\text{non}}(t,\cdot)\right\|_{L^2(\mb{R}^2)}&\lesssim\|u\|_{X(t)}^{p_2}.
	\end{split}
	\end{equation*}
	Finally, combining all of the derived estimate, we can prove
	\begin{equation}\label{important 2D KV}
	\|Nu\|_{X(T)}\lesssim\sum\limits_{k=1}^2\left\|\left(u_0^{(k)},u_1^{(k)}\right)\right\|_{\ml{A}_{m,0}(\mb{R}^2)}+\sum\limits_{k=1}^2\|u\|_{X(T)}^{p_k},
	\end{equation}
	uniform with respect to $T\in[0,\infty)$.
	
	To derive the Lipschitz condition, we can apply H\"older's inequality and the classical Gagliardo-Nirenberg inequality to get
	\begin{equation}\label{improtant 2D KV Lip}
	\left\|Nu-N\bar{u}\right\|_{X(T)}\lesssim\|u-\bar{u}\|_{X(T)}\sum\limits_{k=1}^2\left(\|u\|_{X(T)}^{p_k-1}+\|\bar{u}\|_{X(T)}^{p_k-1}\right),
	\end{equation}
	uniform with respect to $T\in[0,\infty)$.
	
	These derived estimates \eqref{important 2D KV} and \eqref{improtant 2D KV Lip} show the mapping $N:X(T)\rightarrow X(T)$ is a contraction for initial data satisfying
	\begin{equation*}
	\left\|\left(u_0^{(1)},u_1^{(1)}\right)\right\|_{\ml{A}_{m,0}(\mb{R}^2)}+\left\|\left(u_0^{(2)},u_1^{(2)}\right)\right\|_{\ml{A}_{m,0}(\mb{R}^2)}\leq\varepsilon_0,
	\end{equation*}
	with small constant $\varepsilon_0>0$. According to Banach's fixed-point theorem, we complete the proof.
\end{proof}

\section{Treatment of elastic waves with Kelvin-Voigt damping in 3D}\label{3D model section}

In this section we consider the following Cauchy problem for weakly coupled system of semilinear elastic waves with Kelvin-Voigt damping in 3D:
\begin{equation}\label{semi linearproblem Kelvin-Voigt}
\left\{
\begin{aligned}
&u_{tt}-a^2\Delta u-\left(b^2-a^2\right)\nabla\divv u+\left(-a^2\Delta-\left(b^2-a^2\right)\nabla\divv\right)u_t=f(u),&&t>0,\,\,x\in\mb{R}^3,\\
&(u,u_t)(0,x)=(u_0,u_1)(x),&&x\in\mb{R}^3,
\end{aligned}
\right.
\end{equation}
where $b^2>a^2>0$ and the nonlinear terms on the right-hand sides are
\begin{equation*}
f(u):=\left(|u^{(3)}|^{p_1},|u^{(1)}|^{p_2},|u^{(2)}|^{p_3}\right)^{\mathrm{T}}\,\,\,\,\text{with}\,\,\,\,p_1,p_2,p_3>1.
\end{equation*}
For the corresponding linearized problem
\begin{equation}\label{linearproblem Kelvin-Voigt}
\left\{
\begin{aligned}
&u_{tt}-a^2\Delta u-\left(b^2-a^2\right)\nabla\divv u+\left(-a^2\Delta-\left(b^2-a^2\right)\nabla\divv\right)u_t=0,&&t>0,\,\,x\in\mb{R}^3,\\
&(u,u_t)(0,x)=(u_0,u_1)(x),&&x\in\mb{R}^3,
\end{aligned}
\right.
\end{equation}
it allows us to use the Helmholtz decomposition.

To begin with, let us recall the following orthogonal decomposition:
\begin{equation*}
\left(L^2\big(\mb{R}^3\big)\right)^3=\overline{\nabla H^1\big(\mb{R}^3\big)}\oplus \ml{D}_0\big(\mb{R}^3\big),
\end{equation*}
where the space $\overline{\nabla H^1\big(\mb{R}^3\big)}$ denotes the vector fields with divergence zero and $ \ml{D}_0\big(\mb{R}^3\big)$ denotes the vector fields with curl zero (c.f. \cite{Leis1986}).\\
Thus, we can decompose the solution $u=u(t,x)$ to the linearized problem \eqref{linearproblem Kelvin-Voigt} into a potential and a solenoidal part
\begin{equation*}
u=u^{p_0}\oplus u^{s_0},
\end{equation*}
where the vector unknown $u^{p_0}=u^{p_0}(t,x)$ stands for rotation-free and the vector unknown $u^{s_0}=u^{s_0}(t,x)$ stands for divergence-free in a weak sense.

Taking account into the relation $\nabla\divv u=\nabla\times(\nabla\times u)+\Delta u$ in three dimensions, we can decoupled the system \eqref{linearproblem Kelvin-Voigt} into two viscoelastic damped wave equations with different propagation speeds $a$ as well as $b$, respectively,
\begin{equation}\label{Helmoholtz s0}
\left\{
\begin{aligned}
&u^{s_0}_{tt}-a^2\Delta u^{s_0}-a^2\Delta u^{s_0}_t=0,&&t>0,\,\,x\in\mb{R}^3,\\
&\left(u^{s_0},u^{s_0}_t\right)(0,x)=\left(u^{s_0}_0,u^{s_0}_1\right)(x),&&x\in\mb{R}^3,
\end{aligned}
\right.
\end{equation}
and
\begin{equation}\label{Helmoholtz p0}
\left\{
\begin{aligned}
&u^{p_0}_{tt}-b^2\Delta u^{p_0}-b^2\Delta u^{p_0}_t=0,&&t>0,\,\,x\in\mb{R}^3,\\
&\left(u^{p_0},u^{p_0}_t\right)(0,x)=\left(u^{p_0}_0,u^{p_0}_1\right)(x),&&x\in\mb{R}^3.
\end{aligned}
\right.
\end{equation}
The well-posedness of weak solutions to \eqref{Helmoholtz s0} and \eqref{Helmoholtz p0} have been studied in \cite{IkehataTodorovaYordanov2013}, and the well-posedness of distribution solutions have been investigated in \cite{ReissigEbert2018}. The $L^2-L^2$ estimates and $\left(L^2\cap L^m\right)-L^2$ estimates of the solution with $m\in[1,2)$ also have been developed in  \cite{ReissigEbert2018,DabbiccoReissig2014}. Furthermore, $L^p-L^q$ estimates not necessary on the conjugate line of solution to the Cauchy problems \eqref{Helmoholtz s0} or \eqref{Helmoholtz p0} has been investigated in \cite{Ponce1985,Shibata2000}.   Lastly, we mention that the asymptotic profiles of solution with initial data taking from $L^{1,\gamma}\big(\mb{R}^3\big)$ with $\gamma\in[0,1]$ have been studied in \cite{Ikehata2014,Michihisa201801,Michihisa2018}.

To study the semilinear problem \eqref{semi linearproblem Kelvin-Voigt}, we need to derive $\left(L^2\cap L^m\right)-L^2$ estimates and $L^2-L^2$ estimates of solution to the linearized problem \eqref{linearproblem Kelvin-Voigt}. According to the paper \cite{WuChaiLi2017}, one can obtain the next estimates.
\begin{thm}\label{Thm KV 3D energy}
	Let us consider the Cauchy problem \eqref{linearproblem Kelvin-Voigt} with initial data satisfying $\left(u_0^{(k)},u_1^{(k)}\right)\in\ml{A}_{m,s}\big(\mb{R}^3\big)$ for $k=1,2,3$, where $s\geq0$ and $m\in[1,2]$. Then, we have the following estimates:
	\begin{equation*}
	\begin{split}
	&\left\|u^{(k)}(t,\cdot)\right\|_{L^2(\mb{R}^3)}\lesssim(1+t)^{-\frac{6-5m}{4m}}\sum\limits_{k=1}^3\left\|\left(u^{(k)}_0,u^{(k)}_1\right)\right\|_{\ml{A}_{m,0}(\mb{R}^3)}\,\,\,\,\text{if}\,\,\,\,m\in\left[1,\frac{6}{5}\right),\\
	&\left\||D|u^{(k)}(t,\cdot)\right\|_{\dot{H}^s(\mb{R}^3)}+\left\|u_t^{(k)}(t,\cdot)\right\|_{\dot{H}^s(\mb{R}^3)}\lesssim(1+t)^{-\frac{6-3m+2sm}{4m}}\sum\limits_{k=1}^3\left\|\left(u^{(k)}_0,u^{(k)}_1\right)\right\|_{\ml{A}_{m,s}(\mb{R}^3)}.
	\end{split}
	\end{equation*}
\end{thm}

Now, we state our theorem for the global (in time) existence of small data energy solution to \eqref{semi linearproblem Kelvin-Voigt}. To begin with, let us introduce the balanced parameter $\tilde{p}_{\text{bal}}(m)$, $\tilde{\alpha}_1(m)$ and $\tilde{\tilde{\alpha}}_1(m)$ for $m\in\left[1,\frac{6}{5}\right)$ by
\begin{align}
\tilde{p}_{\text{bal}}(m):&=\frac{3+2m}{3-m},\label{sup01}\\
\tilde{\alpha}_{1}(m):&=\frac{m(2+3p_2+p_1p_2)}{2(p_1p_2-1)},\label{sup02}\\
\tilde{\tilde{\alpha}}_{1}(m):&=\frac{m(2+3(p_2+1)p_3+p_1p_2p_3)}{2(p_1p_2p_3-1)}.\label{sup03}
\end{align}
\begin{rem}
	Here we point out the relation between these parameters. If we consider the condition $\tilde{\alpha}_1(m)<\frac{3}{2}$, it also can be rewritten by
	\begin{equation*}
	p_2\left(p_1+1-\tilde{p}_{\text{bal}}(m)\right)>\tilde{p}_{\text{bal}}(m).
	\end{equation*}
	If we consider the condition $\tilde{\tilde{\alpha}}_1(m)<\frac{3}{2}$, it also can be rewritten by
	\begin{equation*}
	p_3\left(p_2\left(p_1+1-\tilde{p}_{\text{bal}}(m)\right)+1-\tilde{p}_{\text{bal}}(m)\right)>\tilde{p}_{\text{bal}}(m).
	\end{equation*}
\end{rem}
\begin{rem}
	From the recent paper \cite{ChenReissig2018}, we remark that the balanced exponent shown in \eqref{sup01} and the balanced parameters shown in \eqref{sup02} as well as \eqref{sup03} correspond to the balanced parameters to the weakly coupled system of semilinear viscoelastic elastic waves in 3D.
\end{rem}

One can follow the procedure of the proof of Theorem 5.5 in \cite{ChenReissig2018} to obtain the following theorem. Without loss of generality, we assume $p_1<p_2<p_3$.

\begin{thm}\label{GESDS02 KV 2D} Let $b^2>a^2>0$ in \eqref{semi linearproblem Kelvin-Voigt} and $m\in\left[1,\frac{6}{5}\right)$. Let us assume $1<p_1<p_2<p_3$, satisfying one of the following conditions:
	\begin{enumerate}
		\item we assume $\tilde{p}_{\text{bal}}(m)<p_1<p_2<p_3\leq 3$;
		\item we assume $\tilde{\alpha}_1(m)<\frac{3}{2}$ if $\frac{2}{m}\leq p_1\leq\tilde{p}_{\text{bal}}(m)<p_2<p_3\leq 3$;
		\item we assume $\tilde{\tilde{\alpha}}_1(m)<\frac{3}{2}$ if $\frac{2}{m}\leq p_1<p_2\leq\tilde{p}_{\text{bal}}(m)<p_3\leq 3$.
	\end{enumerate}
	Then, there exists a constant $\varepsilon_0>0$ such that for any initial data $\left(u_0^{(k)},u_1^{(k)}\right)\in\ml{A}_{m,0}\big(\mb{R}^3\big)$ for $k=1,2,3,$ with
	\begin{equation*}
	\left\|\left(u_0^{(1)},u_1^{(1)}\right)\right\|_{\ml{A}_{m,0}(\mb{R}^3)}+\left\|\left(u_0^{(2)},u_1^{(2)}\right)\right\|_{\ml{A}_{m,0}(\mb{R}^3)}+\left\|\left(u_0^{(3)},u_1^{(3)}\right)\right\|_{\ml{A}_{m,0}(\mb{R}^3)}\leq\varepsilon_0,
	\end{equation*}
	there is uniquely determined energy solution
	\begin{equation*}
	u\in\left(\ml{C}\left([0,\infty),H^1\big(\mb{R}^3\big)\right)\cap\ml{C}^1\left([0,\infty),L^2\big(\mb{R}^3\big)\right)\right)^3
	\end{equation*}
	to \eqref{semi linearproblem Kelvin-Voigt}. Moreover, the solutions satisfies the following estimates:
	\begin{equation*}
	\begin{split}
	\left\|u^{(k)}(t,\cdot)\right\|_{L^2(\mb{R}^3)}&\lesssim(1+t)^{-\frac{6-5m}{4m}+\tilde{\ell}_k}\sum\limits_{k=1}^3\left\|\left(u^{(k)}_0,u^{(k)}_1\right)\right\|_{\ml{A}_{m,0}(\mb{R}^3)},\\
	\left\||D| u^{(k)}(t,\cdot)\right\|_{L^2(\mb{R}^3)}+\left\|u_t^{(k)}(t,\cdot)\right\|_{L^2(\mb{R}^3)}&\lesssim(1+t)^{-\frac{6-3m}{4m}+\tilde{\ell}_k}\sum\limits_{k=1}^3\left\|\left(u^{(k)}_0,u^{(k)}_1\right)\right\|_{\ml{A}_{m,0}(\mb{R}^3)},
	\end{split}
	\end{equation*}
	where
	\begin{equation*}
	0\leq\tilde{\ell}_1=\tilde{\ell}_1(m,p_1):=\left\{
	\begin{aligned}
	&0&&\text{if}\,\,\,\,p_1>\tilde{p}_{\text{bal}}(m),\\
	&\epsilon_0&&\text{if}\,\,\,\,p_1=\tilde{p}_{\text{bal}}(m),\\
	&\tfrac{3-m}{2m}\left(\tilde{p}_{\text{bal}}(m)-p_1\right)&&\text{if}\,\,\,\,p_1<\tilde{p}_{\text{bal}}(m),
	\end{aligned}
	\right.
	\end{equation*}
	\begin{equation*}
	0\leq\tilde{\ell}_2=\tilde{\ell}_2(m,p_2):=\left\{
	\begin{aligned}
	&0&&\text{if}\,\,\,\,p_2>\tilde{p}_{\text{bal}}(m),\\
	&\epsilon_0&&\text{if}\,\,\,\,p_2=\tilde{p}_{\text{bal}}(m),\\
	&\tfrac{3-m}{2m}\left((\tilde{p}_{\text{bal}}(m)-p_1)p_2+(\tilde{p}_{\text{bal}}(m)-p_2)\right)&&\text{if}\,\,\,\,p_2<\tilde{p}_{\text{bal}}(m),
	\end{aligned}
	\right.
	\end{equation*}
	and $\tilde{\ell}_{3}=0$, represent the (no) loss of decay in comparison with the corresponding decay estimates for the solution to the linear Cauchy problem \eqref{linearproblem Kelvin-Voigt} (see Theorem \ref{Thm KV 3D energy}), with $\epsilon_0>0$ being an arbitrary small constant in the limit cases that $p_k=\tilde{p}_{\text{bal}}(m)$ for $k=1,2,3$.	
\end{thm}
\section{Concluding remarks}\label{concluding remarks section}
\begin{rem}
	One may derive asymptotic profiles for elastic waves with Kelvin-Voigt damping in 3D, that is the model \eqref{linearproblem Kelvin-Voigt}. Before doing this, one may apply diagonalization procedure (e.g. \cite{ChenReissig2018}), or asymptotic expansions of eigenvalues and their eigenprojections (e.g. the method used in Section \ref{Section KV 2D spectral}, or \cite{IdeHaramotoKawashima}) to get representation of solutions in the Fourier space. Basing on this representation, one may derive asymptotic profiles in a framework of weighted $L^1$ data by using the tools developed in \cite{Ikehata2014}.
\end{rem}
\begin{rem}
	In Section \ref{global existence small data solution KV 2D} we have proved the global (in time) existence of small data energy solution to \eqref{semilinear problem Kelvin-Voigt 2D}. One also can prove the global (in time) existence of small data Sobolev solution
	\begin{equation*}
	u\in\left(\ml{C}\left([0,\infty),H^{s+1}\big(\mb{R}^2\big)\right)\cap\ml{C}^1\left([0,\infty),H^s\big(\mb{R}^2\big)\right)\right)^2
	\end{equation*}
	to \eqref{semilinear problem Kelvin-Voigt 2D} with initial data taking from $\ml{A}_{m,s}\big(\mb{R}^2\big)$ for $s>0$ and $m\in[1,2)$ by following the next strategy. 
	
	For the high regular data and even not embedded in $L^{\infty}\big(\mb{R}^2\big)$ (i.e. $0<s<1$), we can apply the fractional Gagliardo-Nirenberg inequality, the fractional chain rule, the fractional Leibniz rule (c.f. \cite{RunstSickel1996,Hajaiej2011,Grafakos2014,PalmieriReissig2018}). More precisely, we apply the fractional chain rule to get the estimate for the nonlinear term in Riesz potential spaces $\dot{H}^s\big(\mb{R}^2\big)$ with $s\in(0,1)$. For example,
	\begin{equation*}
	\left\|\big|u^{(2)}(\tau,\cdot)\big|^{p_1}\right\|_{\dot{H}^s(\mb{R}^2)}\lesssim\left\|u^{(2)}(\tau,\cdot)\right\|_{L^{q_1}(\mb{R}^2)}^{p_1-1}\left\|u^{(2)}(\tau,\cdot)\right\|_{\dot{H}^s_{q_2}(\mb{R}^2)},
	\end{equation*}
	where $\frac{p_1-1}{q_1}+\frac{1}{q_2}=\frac{1}{2}$ and $p_1>\lceil s\rceil$. Here, $\lceil \cdot\rceil$ denotes the smallest integer large than a given number, $\lceil s\rceil:=\min\left\{\tilde{s}\in\mb{Z}:s\leq\tilde{s}\right\}$. Then, one can estimate the terms on the right-hand sides by the fractional Gagliardo-Nirenberg inequality.\\
	To estimates the difference between the nonlinearities, we set $g\left(u^{(2)}\right)=u^{(2)}\big|u^{(2)}\big|^{p_1-2}$ to get
	\begin{equation*}
	\big|u^{(2)}(\tau,x)\big|^{p_1}-\big|\tilde{u}^{(2)}(\tau,x)\big|^{p_1}=p_1\int_0^1\left(u^{(2)}(\tau,x)-\tilde{u}^{(2)}(\tau,x)\right)g\left(\nu u^{(2)}(\tau,x)+(1-\nu)\tilde{u}^{(2)}(\tau,x)\right)d\nu.
	\end{equation*}
	Then, applying the fractional Leibniz rule we obtain
	\begin{equation*}
	\begin{split}
	&\left\|\big|u^{(2)}(\tau,\cdot)\big|^{p_1}-\big|\tilde{u}^{(2)}(\tau,\cdot)\big|^{p_1}\right\|_{\dot{H}^s(\mb{R}^2)}\\
	&\lesssim\int_0^1\left\|u^{(2)}(\tau,\cdot)-\tilde{u}^{(2)}(\tau,\cdot)\right\|_{\dot{H}^s_{r_1}(\mb{R}^2)}\left\|g\left(\nu u^{(2)}(\tau,\cdot)+(1-\nu)\tilde{u}^{(2)}(\tau,\cdot)\right)\right\|_{L^{r_2}(\mb{R}^2)}d\nu\\
	&\quad+\int_0^1\left\|u^{(2)}(\tau,\cdot)-\tilde{u}^{(2)}(\tau,\cdot)\right\|_{L^{r_3}(\mb{R}^2)}\left\|g\left(\nu u^{(2)}(\tau,\cdot)+(1-\nu)\tilde{u}^{(2)}(\tau,\cdot)\right)\right\|_{\dot{H}^s_{r_4}(\mb{R}^2)}d\nu,
	\end{split}
	\end{equation*}
	where $\frac{1}{r_1}+\frac{1}{r_2}=\frac{1}{r_3}+\frac{1}{r_4}=\frac{1}{2}$. We next use the fractional Gagliardo-Nirenberg inequality again to estimate all terms on the right-hand side. Thus, after choosing suitable parameters $q_1,q_2,r_1,r_2,r_3,r_4$, a new lower bound $1+\lceil s\rceil$ for the exponent $p_1$ comes. 
	
	For the large regular initial data with $s>1$, it allows us to use the fractional powers (c.f. \cite{DabbiccoEbertLucente}) and the continuous embedding $H^s\big(\mb{R}^2\big)\hookrightarrow L^{\infty}\big(\mb{R}^2\big)$. At this time, we need to give a new lower bound $1+s$ for the exponents.
\end{rem}
\begin{rem}
	In Theorem \ref{GESDS02 KV 2D} we only show the global existence result for the energy solution to \eqref{semi linearproblem Kelvin-Voigt} with initial data belonging to $\ml{A}_{m,0}\big(\mb{R}^3\big)$ for $m\in\left[1,\frac{6}{5}\right)$. If one is interested in initial data taking from $\ml{A}_{m,s}\big(\mb{R}^3\big)$ for all $m\in[1,2)$ and $s\geq0$, one can read Section 5 of the recent paper \cite{ChenReissig2018}.
\end{rem}

\section{Acknowledgments}
	The PhD study of the author is supported by S\"achsiches Landesgraduiertenstipendium. The author thank his supervisor Michael Reissig for the suggestions in the preparation of the final version and the anonymous referees for carefully reading the paper.

%%%%%%%%%%% The bibliography starts:

%%  ios1.bst will be used to                               %%

\nocite{*} 
% if your bibliography is in bibtex format, use those commands:
%\bibliographystyle{ios1}           % Style BST file.
%\bibliography{bibliography}        % Bibliography file (usually '*.bib')

% or include bibliography directly:

% ------------------------------------------------------------------------
\end{document}